\documentclass[12pt]{article}
\usepackage[utf8]{inputenc}
\usepackage{amssymb,amsfonts,amsmath,amsthm,epstopdf,graphicx}
\usepackage{appendix}
\usepackage{color, ulem, float}
\usepackage{url}
\usepackage{subfigure}
\usepackage{bbold}

\newcommand{\witi}{\widetilde}
\newcommand{\fracd}[2]{\frac {\displaystyle #1}{\displaystyle #2 }}
\newcommand{\zz}{{\mathbb Z}}
\newcommand{\nn}{{\mathbb N}}
\newcommand{\RR}{{\mathbb R}}

\newcommand{\cald}{{\mathcal D}}

\newcommand{\veps}{\varepsilon}

\newcommand{\vxin}{\overrightarrow{\xi_{_N}}}
\newcommand{\xna}{X^{(N,1)}}
\newcommand{\xnb}{X^{(N,2)}}
\newcommand{\xnc}{X^{(N,3)}}
\newcommand{\xn}{X^{(N)}}
\newcommand{\tn}{T^{ (N)}}

\newcommand{\yn}{Y^{(N)}}
\newcommand{\pn}{P^{(N)}}
\newcommand{\qn}{Q^{(N)}}
\newcommand{\tyn}{{\witi Y}^{(N)}}
\newcommand{\tpn}{{\witi P}^{(N)}}

\newcommand{\beq}{\begin{eqnarray*}}
\newcommand{\feq}{\end{eqnarray*}}
\newcommand{\beqn}{\begin{eqnarray}}
\newcommand{\feqn}{\end{eqnarray}}
\newcommand{\bec}{\begin{claim}}
\newcommand{\fec}{\end{claim}}
\newcommand{\becn}{\begin{claim*}}
\newcommand{\fecn}{\end{claim*}}

\newtheorem{theorem}{Theorem}
\makeatletter \@addtoreset{theorem}{section}\makeatother

\newtheorem{assume}[theorem]{Assumption}
\newtheorem{assumep*}{Assumption~1'}
\newtheorem{proposition}[theorem]{Proposition}
\newtheorem{lemma}[theorem]{Lemma}
\newtheorem{corollary}[theorem]{Corollary}
\newtheorem{remark}[theorem]{Remark}

\title{Moran-type bounds for the fixation probability in a frequency-dependent Wright-Fisher model}

\author{Timothy Chumley\thanks{Department of Mathematics, Iowa State University, Ames, IA 50011, USA; e-mail: tchumley@iastate.edu}
\and Ozgur Aydogmus\thanks{Social Sciences University of Ankara, Department of Economics,  H\"{u}k\"{u}met Meydani No:2 Ulus, \qquad Ankara, Turkey; e-mail: ozgur.aydogmus@asbu.edu.tr}
        \and Anastasios Matzavinos\thanks{Division of Applied Mathematics, Brown University, Providence,
         RI 02912, USA, and\qquad \qquad \qquad \qquad
            Computational Science and Engineering Laboratory, ETH Z\"{u}rich, CH-8092, Z\"{u}rich, Switzerland; \qquad \qquad \quad e-mail: matzavinos@brown.edu} \and Alexander Roitershtein\thanks{Department of Mathematics, Iowa State University, Ames, IA 50011, USA; e-mail: roiterst@iastate.edu}
}
\date{}
\begin{document}
\maketitle
\begin{abstract}
We study stochastic evolutionary game dynamics in a population of finite size.
Individuals in the population are divided into two dynamically evolving groups.
The structure of the population is formally described by a Wright-Fisher type Markov chain with a frequency dependent fitness.
In a strong selection regime that favors one of the two groups, we obtain qualitatively matching lower and upper bounds
for the fixation probability of the advantageous population. In the infinite population limit we obtain an exact result showing that
a single advantageous mutant can invade an infinite population with a positive probability. We also give asymptotically sharp bounds for the
fixation time distribution.
\end{abstract}

\section{Introduction}
\label{intro}
Evolutionary game theory \cite{HS,MSPbook,nowak2006,econegt,weibull} is a mathematically accessible way of modeling the evolution of populations consisting of groups of individuals which perform different forms of behavior. It is commonly assumed within this theoretical framework that individuals reproduce or adopt their behavior according to their fitness, which depends on the population composition through a parameter representing utility of a random interaction within the population.
The fundamental interest of the theory is in understanding which forms of behavior have the ability to persist and which forms have a tendency to be driven out by others.
\par
In the language of game theory, behavior types are called strategies and the utility is identified with the expected payoff in an underlying game.
The basic biological interpretation is that each strategy is linked to a phenotype in the population and more successful types of behavior have higher reproductive fitness. In applications to the evolution of social or economic behavior, the propagation of strategies can be explained by an interplay between cultural inheritance, learning, and imitation \cite{structure_popula,econegt,PNAS_imitation,PNAS_imitation1}.
\par
In the case of finite populations, the evolutionary dynamics is typically modeled by a discrete-time Markov process such as Moran and Wright-Fisher processes  \cite{fudenbergss,nowakWF,sabin,nowakmoran,traulsen}. In this paper we focus on the evolutionary game dynamics of the Wright-Fisher process introduced by Imhof and Nowak in \cite{nowakWF}. In this discrete time model, there are two competing types of individuals in a population of fixed size whose fitness depends, up to a certain selection parameter, on the composition of the population (type frequencies). During each generation every individual is replaced by an offspring whose type is determined at random, independently of others, based on the fitness profile of the population. The resulting model is a discrete-time Markov chain whose states represent the number of individuals of one of the two types present in the current generation.
\par
The Markov chain has two absorbing states corresponding to the situation where one of the two types becomes extinct and the other invades the population. The study of the probability of fixation in an absorption state representing a homogeneous population thus becomes a primary focus of the theory \cite{measuress,sabin,formal_properties,evol_fixation,wahl}. Following common jargon in the literature, we will occasionally refer to the evolution of the Markov process until fixation as the {\it invasion dynamics} of the model.
\par
Formally, the Wright-Fisher process introduced in \cite{nowakWF} can be seen as a variation of the classical Wright-Fisher model for genetic drift \cite{dgbook,stflourgene,ewens,gale} with a frequency-dependent selection mechanism. Throughout the paper we are concerned with the generic case of the selection that systematically favors one of the two population types. We thus will impose a condition ensuring that the {\it local drift} of the Markov chain (i.\,e., the expected value of the jump conditioned on the current state of the Markov chain) is strictly positive at any state. Since the fitness in Imhof and Nowak's model is determined by the payoff matrix of a game, this condition turns out to be essentially equivalent to the assumption that one of the strategies in the
underlying $2\times 2$ game is dominant.
\par
The main goal of this paper is to derive similar upper and lower bounds for the fixation probability of the model.
The bounds become sharp in the limit of the infinite population, but even for a fixed population size they have
similar mathematical form and thus capture adequately the invasion dynamics of the model. The core of the paper is an exploration of Moran's method \cite{mbounds} and its ramifications in our framework.
\par
The rest of the paper is organized as follows. The underlying model is formally introduced in Section~\ref{model}. Our results are stated and discussed in Section~\ref{results}. Conclusions are outlined in Section \ref{conclusion}. The proofs are deferred to Section~\ref{proofs}. Finally, an analogue of our main result for a  related frequency-dependent Moran model, which has been introduced in \cite{nature04a,nowakmoran}, is briefly discussed in the Appendix.
\section{Mathematical model}
\label{model}
The Wright-Fisher Markov process introduced in \cite{nowakWF}, which we now present, describes the evolution of two competing types of individuals in a population of fixed size $N$ evolving in discrete non-overlapping generations. Individuals of the first type follow a so-called strategy $A$ and those of the second type follow a so-called strategy $B$.  The underlying $2\times 2$ symmetric game is described by the payoff matrix
\beq
\begin{tabular}{ l | c  r }
     & A & B \\ \hline
    A & a & b \\
    B & c & d \\
  \end{tabular}
\feq
where $a,b,c,d$ are given positive constants. The matrix entries $a,b,c,d$ represent the utility of an interaction of (the individual of type) $A$ with $A,$ $A$ with $B,$ $B$ with $A,$ and $B$ with $B,$ respectively, for the first named individual in the pair.
\par
We denote by $\xn_t$ the number of individuals following strategy $A$ in the generation $t \in \zz_+.$ Here and henceforth $\zz_+$ stands for the set of nonnegative integers $\nn\cup\{0\}.$
With each of the two strategies is associated a {\it fitness}, which, when $\xn_t=i$, is given respectively by
\beq
f_{_N}(i)=1-w+w\pi_{_A}(i,N)\quad \mbox{and}\quad g_{_N}(i)=1-w+w\pi_{_B}(i,N),
\feq
where $w\in[0,1]$ is the so called {\it selection parameter}, while
\beq
\pi_{_A}(i,N)=\frac{a(i-1)+b(N-i)}{N-1}
\quad \mbox{and}\quad
\pi_{_B}(i,N)=\frac{ci+d(N-i-1)}{N-1}
\feq
are expected payoffs in a single game with a randomly chosen, excluding self-interaction, member of the population. The selection parameter $w$ is a proxy for modeling the strength of the effect of interactions (governed by the payoff matrix) on the evolution of the population compared to inheritance.
\par
Given that $\xn_t=i$, the number of individuals in the next generation adopting strategy $A$ is described by $N$ independent Bernoulli trials with success probability given by
\beqn
\label{xin}
\xi_{_N}(i)=\frac{if_{_N}(i)}{if_{_N}(i)+(N-i)g_{_N}(i)}.
\feqn
Thus, conditionally on $\xn_t,$ the next generation $\xn_{t+1}$ is a binomial random variable \\ $BIN\bigl(N,\xi_{_N}(\xn_t)\bigr)$ centered around $N\xi_{_N}(\xn_t):$
\beqn
\label{wf}
P\bigl(\xn_{t+1}=j|\xn_t=i\bigr)={N\choose j}\bigl(\xi_{_N}(i)\bigr)^j\bigl(1-\xi_{_N}(i)\bigr)^{N-j}
\feqn
for all $0\leq i,j\leq N.$ The Wright-Fisher Markov chain with transition kernel in the form \eqref{wf} comprises a class of classical models of population genetics \cite{dgbook,stflourgene,ewens,gale}. Under the assumptions on the underlying game stated below in this section, the most mathematically related case is a model of genetic drift in a diploid population with favorable selection and without mutation, which formally corresponds to choosing a payoff matrix with $a=b>1$ and $c=d=1.$
\par
Note that the Markov chain has two absorbing states, $0$ and $N$, which correspond to the extinction of individuals using one of the two strategies. Our primary objective in this paper focuses on the estimation of the following fixation (absorption at $N,$ or yet alternatively, invasion) probability:
\beqn
\label{pni}
p_{_N}(i)=P\bigl(\xn_T=N|\xn_0=i\bigr),
\feqn
where
\beqn
\label{ti}
T=\min\bigl\{t\in\nn:\xn_t=0~\mbox{or}~\xn_t=N\bigr\}.
\feqn
Throughout the paper we are interested in the dependence of $p_{_N}(i)$ on $i$ and $N$ while the parameters
$a,b,c,d$ and $w>0$ are maintained fixed. We make the following standing assumption.
For an integer $N\geq 2,$ let $\Omega_{_N}$ and $\Omega_{_N}^o$ denote the state space and the set of transient (non-absorbing) states of $\xn,$ respectively. That is,
\beqn
\label{omegan}
\Omega_{_N}=\{0,1,\ldots,N\}\qquad \mbox{and}\qquad \Omega_{_N}^o=\{1,\ldots,N-1\}.
\feqn
\begin{assume}
\label{assume1}
There exist an integer $N_0\geq 2$ and real constants $\alpha$ and $\gamma$ such that
\beq
0<\alpha\leq\frac{g_{_N}(i)}{f_{_N}(i)}\leq \gamma<1 \quad \mbox{for all}\quad N\geq N_0~\mbox{and}~i\in\Omega_{_N}^o.
\feq
\end{assume}
A connection between Assumption~\ref{assume1} and the structure of the local drift of the Markov chain $\xn$, which is crucially important for our agenda, is discussed
in Section~\ref{finite}. In the rest of this section we focus on more technical aspects and immediate implications of this assumption.
\par
Note that for a fixed $N\in\nn,$ both $f_{_N}(i)$ and $g_{_N}(i)$ are linear functions (first-order polynomials) of $i,$ and hence
their graphs are straight segments. Thus the above assumption implies that
\begin{align}
\nonumber
\label{edgeu}
&
a=\lim_{N\to\infty}f_{_N}(N-1)=\lim_{N\to\infty}f_{_N}(N)>
\\
&
\qquad
\qquad
\qquad
>b=\lim_{N\to\infty}g_{_N}(N-1)=\lim_{N\to\infty}g_{_N}(N)
\end{align}
and
\beqn
c=\label{edgel}
\lim_{N\to\infty}f_{_N}(1)=\lim_{N\to\infty}f_{_N}(0)>d=\lim_{N\to\infty}g_{_N}(1)=\lim_{N\to\infty}g_{_N}(0).
\feqn
 Hence $A$ is the strictly dominant strategy in the underlying game.
In particular, both players implementing strategy $A$ is the unique Nash equilibrium of the game and also the unique evolutionary stable strategy \cite{HS,weibull}.
\par
It has been pointed out in \cite{nowakmoran} (for a related frequency-dependent Moran process which we discuss in the Appendix) and \cite{nowakWF} (for the Wright-Fisher process considered in this paper) that, to a large extent, the invasion dynamics of the model for a given population size $N$ can be characterized by the behavior of the sign of the function $h_{_N}(i)=f_{_N}(i)-g_{_N}(i),$ and that the latter is entirely determined in the whole range $i\in  \Omega_{_N}^o$ by the pair $\bigl\{\mbox{sign}\bigl(h_{_N}(1)\bigr),\mbox{sign}\{h_{_N}\bigl(N-1)\bigr)\bigr\}.$ Our Assumption~\ref{assume1} can be thought of as a uniform over $N$ variation of the $A$-dominance condition $h_{_N}(0)>0$ and $h_{_N}(N)>0$ introduced in \cite{nature04a,nowakmoran} ($\xi'>0$ and $\zeta'<0$ in their notation). In fact, using the linear structure of $f_{_N}(i)$ and $g_{_N}(i),$ it is straightforward to verify that Assumption~\ref{assume1} is equivalent to the following one ($w>0$ along with $\xi>0$ and $\zeta<0$ in notation of \cite{nowakmoran}):
\begin{assumep*}
\label{assume1p}
$w>0,$ $a>b,$ $c>d.$
\end{assumep*}
Furthermore, if the hypotheses of Assumption~1' are satisfied, one can set
\beq
N_0=\min\bigl\{N\geq 2:a(N-1)>cN-d~\,\mbox{and}~\,b(N-1)>d(N-2)+c\bigr\},
\feq
and
\beqn
\label{alphagamma}
\begin{split}
\alpha&=\min\Bigl\{\frac{g_{_{N_0}}(1)}{f_{_{N_0}}(1)},\frac{g_{_{N_0}}(N_0-1)}{f_{_{N_0}}(N_0-1)},\lim_{N\to\infty}\frac{g_{_N}(1)}{f_{_N}(1)},
\lim_{N\to\infty}\frac{g_{_N}(N-1)}{f_{_N}(N-1)}\Bigr\},
\\
\gamma&=\max\Bigl\{\frac{g_{_{N_0}}(1)}{f_{_{N_0}}(1)},\frac{g_{_{N_0}}(N_0-1)}{f_{_{N_0}}(N_0-1)},\lim_{N\to\infty}\frac{g_{_N}(1)}{f_{_N}(1)},
\lim_{N\to\infty}\frac{g_{_N}(N-1)}{f_{_N}(N-1)}\Bigr\},
\end{split}
\feqn
in order to obtain $\alpha,\gamma$ and $N_0$ postulated by Assumption~\ref{assume1} explicitly, in terms of the basic data $a,b,c,d$ and $w.$
\par
A canonical example of a game satisfying Assumption~\ref{assume1} is the prisoner's dilemma where $b>d>a>c.$ Examples of games with $b>a>c>d$ ({\it Time's sales versus Newsweek's sales} and {\it cheetahs and antelopes}, respectively) are considered in \cite[Section~3.2]{timevsnewsweek} and \cite[Section~3.2]{gtalive}.
Another version of the {\it cheetahs and antelopes} discussed in Section~3.2 of \cite{gtalive} provides an example of a game where $b>c>a>d.$  The invasion dynamics for two examples of a frequency-dependent Moran process with an underlying game such that  $a>c>b>d$ is analyzed in \cite{nowakmoran}.
\section{Results and discussion}
\label{results}
The organization of this section is as follows. The section is divided into four subsections, the first two contain a preliminary discussion and the other two present our main results. Section~\ref{prelim} aims to provide a brief background and a suitable general context for our approach and results. It also contains a summary of our main results. Section~\ref{neutral} discusses relevant results of \cite{nowakWF} from the perspective outlined in Section~\ref{prelim}. In Section~\ref{finite} we state our results for a fixed finite population size $N,$ and in Section~\ref{branching} we discuss their asymptotic counterparts in the infinite population limit and some implications.
\subsection{Preliminary discussion}
\label{prelim}
In \cite{nowakWF} it is shown, among other results, that under Assumption~\ref{assume1} the selection in the Wright-Fisher process
favors $A$ in that $p_{_N}(i) > i/N$ for all $i\in\Omega_{_N}^o.$ Our goal is to obtain a further insight into the behavior of the fixation probability $p_{_N}(i)$ as a function of $i$ and $N.$
Our main contribution can be partially summarized as follows:
\begin{theorem}
[Main results in a nutshell]
\label{main1a}
Let Assumption~\ref{assume1} hold. Then, for the Wright-Fisher process defined in \eqref{xin} and \eqref{wf},
there exist constants $\rho, \theta \in (0,1)$ such that
\beqn
\label{result}
\frac{1-\rho^i}{1-\rho^N}\leq p_{_N}(i)\leq \frac{1-\theta^i}{1-\theta^N},\qquad N\geq N_0,\,i\in \Omega_{_N}^o.
\feqn
Furthermore, there exists a constant $q\in [\rho,\theta]$ such that
\beqn
\label{result1}
\lim_{N \to \infty}p_{_N}(i)=1-q^i,\qquad \forall~i\in\nn.
\feqn
\end{theorem}
The first part of the above theorem is the content of Theorem~\ref{main1}, and the second one in a more detailed form is stated in Theorem~\ref{newt}.
Notice that the limit in \eqref{result1} has the same form $1-c^i$ for some $c\in(0,1)$ as the asymptotic of the lower and upper bounds in \eqref{result}.
In intuitive accordance with the fact that $A$ is the unique Nash equilibrium and evolutionary stable strategy in the underlying game,
\eqref{result1} implies that a single advantageous mutant has a non-zero probability of invasion even at the infinite population limit.
\par
The Wright-Fisher Markov chains with directional selection (i.\,e. favoring one of the two population types), either the classical one with constant selection
bias $\frac{g_{_N}(i)}{g_{_N}(i)}$ or the more general one with the frequency dependent selection mechanism introduced in \cite{nowakWF}, are known to exhibit complex multi-scale dynamics. We refer to \cite[Section~6.3.1]{dgbook}, \cite[Section~1.4.3]{ewens}, \cite[Section~3]{discrete}, and \cite{scale,scale1} for a description of several possible scaling schemes and limiting procedures for these models. Due to the high likelihood of big jumps and space-wise inhomogeneity of the local drift, obtaining a quantitative insight into a mechanism of transforming a relatively simple structure of the transition kernel into the global dynamics of these Markov chains turns out to be a challenging task. The usual approach to overcome the difficulty is to implement a small parameter expansion which,  although often formally not limited to any particular drift structure, ultimately leads to the (conceptually undesirable) comparison of the model to stochastic processes without drift.
\par
When $w$ is small---a regime commonly referred to as \textit{weak selection}---the fitnesses of individuals play little role on the dynamics of evolution, and rather it is simply the proportions of the previous generation which play the primary role in determining the next generation.
\par
The main content of our work is an alternative method, which is in essence equivalent to an indirect coupling of the Imhof-Nowak model with an ``exponential submartingale". In particular, this method allows us to considerably improve the previous work and obtain
qualitatively matching lower and upper bounds for the fixation (absorption) probabilities in this model.
\par
Using explicit formulas available for the Moran chain, it can be verified that if \\ $\frac{g_{_N}(i)}{f_{_N}(i)}=q\in (0,1)$ for all $N\in\nn$ and
$i\in\Omega_{_N}^o(i),$ then
\beqn
\label{exp}
p_{_N}(i)=\frac{1-q^{i}}{1-q^{N}}.
\feqn
This example formally corresponds to the payoff matrix $a=b=\frac{1}{w}\frac{1-q}{q}$ and $c=d=1.$
The expression for the absorbtion probability in the form $p_{_N}(i)=\frac{1-q^i}{1-q^N}$ is universal for finite-state Markov chains $X_t$ with
a space homogenous transition kernel and positive average drift for which a constant $q\in (0,1)$ can be found such that $M_t=q^{X_t}$ is a martingale.
Fixation probabilities for the diffusion approximation of the classical Wright-Fisher model with selection are in the form
$p(x)=\frac{1-\rho^x}{1-\rho},$ which is similar. This form of the fixation probability is believed to be universal for a large class of  evolutionary models in structured populations \cite{nowakfix}. Results similar to \eqref{exp} for the Wright-Fisher process were obtained by Kimura via diffusion approximations (see also \cite{ewens}). However, these results are only valid for large population sizes and Markov chains that move by small steps, i.e. $w=1-r$ is sufficiently small.
\par
The proof of the bounds for a given population size relies on identifying ``exponential sub- and super-martingales" dominating the process (we remark, for instance, that for a linear Brownian motion $B_t+\mu t,$ $\mu>0,$ the proper choice is $\rho=e^{-2\mu}$ by virtue of Theorem~8.5.6 in \cite{durrett}).
\subsection{Comparison to the neutral Wright-Fisher process}
\label{neutral}
In the regime of neutral selection where $w=0$, $\xn_t$ is a martingale, and the precise computation $p_{_N}(i) = i/N$ readily follows.
When $w > 0$ the situation is more difficult due to the space inhomogenity of the transition kernel combined with the large amplitude of one-step fluctuations.
The latter is compared here to the nearest-neighbor transitions of a Moran process, for instance the companion Moran model introduced in \cite{nature04a,nowakmoran}.
By using a comparison with the neutral selection case, it is shown in \cite{nowakWF} that if Assumption~\ref{assume1} is satisfied, then
\beqn
\label{iovern}
p_{_N}(i)> i/N,\qquad \forall~N\geq N_0,\,i\in\Omega_{_N}^o.
\feqn
This result holds for any $w>0$ and is an instance of the following general principle.
\begin{proposition}
\label{prop}
Let $X=(X_t)_{t\in \zz_+}$ be a Markov chain on $\Omega_{_N}$ for some $N\geq 2.$ For $i\in\Omega_{_N},$ let
\beqn
\label{ldrift}
\mu(i)=E(X_{t+1}-X_t|X_t=i)
\feqn
denote the local drift of $X$ at site $i.$ Suppose that:
\begin{enumerate}
\item $0$ and $N$ are absorbing states.
\item If $i\in\Omega_{_N}^o$ and $j\in\Omega_{_N},$ then $P(X_m=j|X_0=i)>0$ for some $m\in\nn.$
\item $\mu(i)\geq 0$ for any $i\in\Omega_{_N}^o.$
\end{enumerate}
Let $p_{_N}(i)=P(X~\mbox{\rm absorbs at}~N|X_0=i).$ Then $p_{_N}(i)\geq i/N$ for any $i\in\Omega_{_N}^o.$ Furthermore, the
inequality is strict if and only if $\mu(i)>0$ for at least one site $i\in\Omega_{_N}^o.$
\end{proposition}
Note that under Assumption~\ref{assume1}, we have the following relation for the local drift of the Wright-Fisher Markov chain $\xn$
\begin{align}
\nonumber
\label{drift}
\mu_{_N}(i)&:=E\bigl(\xn_{t+1}-\xn_t\bigl|\xn_t=i\bigr)=N\xi_{_N}(i)-i
\\
&
=\frac{i(N-i)\bigl(f_{_N}(i)-g_{_N}(i)\bigr)}{if_{_N}(i)+(N-i)g_{_N}(i)}> 0,\qquad \qquad \forall~i\in\Omega_{_N}^o.
\end{align}
The proof of the proposition is in essence the observation that $X$ is a bounded submartingale, and hence
$E(X_T|X_0)=Np_{_N}(X_0)\geq X_0$ by the optional stopping theorem \cite[Theorem~5.7.5]{durrett}.
\par
In the weak selection regime, \cite{nowakWF} provides a nearly complete analysis of $p_{_N}(i)$ and in particular obtains
a version of the so-called one-third law of evolutionary dynamics for the model. The results of \cite{nowakWF} for the fixation probability under weak selection are further refined and extended in \cite{sabin}. In particular, \cite{sabin} derives a second order correction term to $i/N$ for the fixation probability $p_{_N}(i).$ In this paper, we concentrate on the case of directional (beneficial for type $A$) selection postulated in Assumption~\ref{assume1}, but we do not make the assumption of weak selection.
\par
We conclude this subsection with an interpretation of the drift $\mu_{_N}(i)$ which will not be used in the rest of the paper, but we believe
it is of interest on its own. For $i\in \Omega_{_N}$ and $j=1,\ldots,N,$ let
\beq
S_{N,i}(j)=
\left\{
\begin{array}{lll}
1&\mbox{if}&j\leq i\\
0&\mbox{if}&j> i
\end{array}
\right.
\quad
\mbox{\rm and}
\quad
F_{N,i}(j)=
\left\{
\begin{array}{lll}
f_{_N}(i)&\mbox{if}&j\leq i\\
g_{_N}(i)&\mbox{if}&j>i.
\end{array}
\right.
\feq
Thus, if one enumerates and in addition also labels the individuals at the state $\xn_t=i$ in such a way that the first $i$ individuals are of type $A$ and get label $1,$ and the remaining $N-i$ individuals are of type $B$ and get label $0,$ then $S_{N,i}(j)$ and $F_{N,i}(j)$ represent, respectively,
the label and the fitness of the $j$-th individual. Further, using the above enumeration, let $(u,v),$ $u<v,$ be a pair of the individuals chosen at random. That is,
\beq
P(u=j,v=k)=\frac{2}{N(N-1)}\quad \mbox{for any}\quad j,k\in \Omega_{_N}\backslash\{0\},\, j<k.
\feq
Let
\beq
H_{N,i}:=\frac{2i(N-i)}{N(N-1)}=E\bigl(S_{N,i}(v)-S_{N,i}(u)\bigr)
\feq
be the {\it heterozygosity} \cite[Section ~1.2]{dgbook} of the Wright-Fisher process $\xn$ at state $i\in\Omega_{_N}^o,$ that is the probability that two individuals randomly chosen from the population when $\xn_t=i$ have different types. In this notation,
\beq
\mu_{_N}(i)&=&\frac{N(N-1)}{2}H_{_N}(i)\frac{f_{_N}(i)-g_{_N}(i)}{if_{_N}(i)+(N-i)g_{_N}(i)}
\\
&=&
\frac{N(N-1)}{2}\cdot \frac{E\bigl(F_{N,i}(v)-F_{N,i}(u)\bigr)}{\sum_{j=1}^N F_{N,i}(j)}
\\
&=&
\frac{1}{2}\frac{\sum_{j,k=1}^N \bigl|F_{N,i}(k)-F_{N,i}(j)\bigr|}{\sum_{j=1}^N F_{N,i}(j)},
\feq
suggesting that the drift $\mu_{_N}(i)$ can serve as a measure of heterozygosity suitable for our game-theoretic framework.
\subsection{Moran's bounds and a coupling with the classical Wright-Fisher chain}
\label{finite}
Intuitively, it is clear that in the framework of Proposition~\ref{prop}, the local drift $\mu(i)$ defined in \eqref{ldrift} is a characteristic of the Markov chain $X_t$ which is intimately related to the value of the fixation probabilities. Notice, for instance, that if $X_0=i$ and $T$ is the absorbtion time of the Markov chain $X,$
then
\beq
\sum_{t=0}^\infty E\bigl(\mu_{_N}(X_t)\bigr)=E(X_T-X_0)=Np_{_N}(i)-i.
\feq
The general heuristic assertion of a close association between the local drift and the fixation probabilities is especially
evident in the particular instance of the Wright-Fisher process, since according to \eqref{drift},
\beqn
\label{xi-mu}
\xi_{_N}(i)=\frac{i}{N}+\frac{1}{N}\mu_{_N}(i),
\feqn
and, by virtue of \eqref{wf}, $\bigl(\xi_{_N}(i)\bigr)_{i\in \Omega_{_N}^o}$ is the sequence defining the dymamics of the model. Thus, in view of the inequality in \eqref{drift}, in order to study the shape of $p_{_N}(i)$ as a function of $i$ it may be conceptually desirable to compare $\xn$ with a suitable stochastic process with positive drift, for which the solution to the gambler's ruin problem is explicitly known.
\par
It has been emphasized in the work of \cite{nowakWF,nowakmoran} that the fitness difference $h_{_N}(i)=f_{_N}(i)-g_{_N}(i),$ and in particular its sign,
is a major factor influencing the invasion dynamics of the model. Note that by virtue of \eqref{drift}, the sign of $h_{_N}(i)$ coincides with the sign of the local drift $\mu_{_N}(i).$ Moreover, \eqref{drift} can be rewritten as
\beqn
\label{cdrift}
\mu_{_N}(i)=\frac{i(N-i)\bigl(1-\frac{g_{_N}(i)}{f_{_N}(i)}\bigr)}{i+(N-i)\frac{g_{_N}(i)}{f_{_N}(i)}}, \qquad \forall~i\in\Omega_{_N}^o,
\feqn
showing that the value of  $\mu_{_N}(i)$ is in fact determined by the ratio $g_{_N}(i)\big\slash f_{_N}(i).$ From this perspective, Assumption~\ref{assume1} together with \eqref{alphagamma} can be thought of as a tool establishing lower and upper bounds for the drift in terms of the selection parameter $w$ and payoff matrix of the underlying game. In fact, \eqref{cdrift} yields
\beqn
\label{mdrift}
\frac{i(N-i)(1-\gamma)}{i+(N-i)\gamma} \leq \mu_{_N}(i)\leq \frac{i(N-i)(1-\alpha)}{i+(N-i)\alpha}, \qquad \forall~i\in\Omega_{_N}^o,
\feqn
where the lower and upper bounds have the form of the local drift of the Wright-Fisher process with a constant selection $\frac{g_{_N}(i)}{f_{_N}(i)}.$
These bounds suggest in particular the possibility of a comparison of our model with the classical Wright-Fisher process of mathematical genetics,
and furthermore indicate, at least at the level of heuristic argument, that the dynamics of $\xn$ should be similar to that of the Wright-Fisher process with constant selection.
\par
The following lemma, whose proof is included in Section~\ref{proof-lem2}, is the key technical observation we use to derive Theorem~\ref{main1}, our main result regarding the fixation probability in finite populations.
\begin{lemma}
\label{lem2}
Let Assumption~\ref{assume1} hold. Then:
\begin{itemize}
\item[(a)]
There exists a constant $\rho\in (0,1)$ such that
$E\bigl(\rho^{\xn_{t+1}}|\xn_t=i\bigr)\leq \rho^i$ for any $N\geq N_0,$ $i\in\Omega_{_N}^o$ and integer $t\in \zz_+.$
\item[(b)]
There exists a constant $\theta\in (0,1)$ such that
$E\bigl(\theta^{\xn_{t+1}}|\xn_t=i\bigr)\geq \theta^i$ for any $N\geq N_0,$ $i\in\Omega_{_N}^o$ and integer $t\in \zz_+.$
\item[(c)] Furthermore, in the above conclusions one can choose, respectively,
\beqn
\label{parameters}
\rho=e^{-2(1-\gamma)}\qquad\mbox{and}\qquad \theta=e^{-\frac{2(1-\alpha)}{\alpha}}.
\feqn
\end{itemize}
\end{lemma}
The lemma is a suitable modification of the ideas of Moran \cite{mbounds,mgeneral} adapted to the present model with a frequency-dependent selection. The values of $\rho$ and $\theta$ suggested in \eqref{parameters} are obtained from the corresponding estimates of Moran for the classical Wright-Fisher chain with a constant
selection \cite{mbounds}. Slightly better bounds can be obtained based on Arnold's elaboration of Moran's original approach \cite{arnold,buckley_seneta}.
\par
Lemma~\ref{lem2} implies that the sequence $\bigl(\rho^{\xn_t}\bigr)_{t\in \zz_+}$ is a submartingale while $\bigl(\theta^{\xn_t}\bigr)_{t\in \zz_+}$ is a supermartingale.
Since both $\rho^{\xn_t}$ and $\theta^{\xn_t}$ are non-negative random variables bounded from above by one,
Doob's optional stopping theorem \cite[Theorem~5.7.5]{durrett} implies that with probability one,
\beq
\rho^{\xn_0}\geq E\bigl(\rho^{\xn_T}|\xn_0\bigr)=p_{_N}(\xn_0)\rho^N+\bigl(1-p_{_N}(\xn_0)\bigr)
\feq
and
\beq
\theta^{\xn_0}\leq E\bigl(\theta^{X_T}|\xn_0\bigr)=p_{_N}(\xn_0)\theta^N+\bigl(1-p_{_N}(\xn_0)\bigr).
\feq
This yields the following exponential bounds of the form \eqref{exp} for the fixation probabilities.
\begin{theorem}
\label{main1}
Suppose that Assumption~\ref{assume1} holds. Let constants $\rho\in (0,1)$ and $\theta\in (0,1)$ be determined by \eqref{parameters}. Then
\beq
\frac{1-\rho^i}{1-\rho^N}\leq p_{_N}(i)\leq \frac{1-\theta^i}{1-\theta^N},
\feq
for all $N\geq N_0$ and $i\in\Omega_{_N}^o.$
\end{theorem}
Since $\frac{1-\rho^N}{N}<\frac{1-\rho^i}{i}$ for all $i\in\Omega_{_N}^o,$ the linear bound in \eqref{iovern} can be recovered as a direct implication of Theorem~\ref{main1}. The upper bound suggested in the theorem indicates that the exponential lower bound captures correctly the
qualitative behavior of $p_{_N}(i)$ as a function of the initial state $i.$
\par
Theorem~\ref{main1} can be strengthened to the following coupling result. In what follows we refer to a Markov chain on $\Omega_{_N}$ with transition kernel given by \eqref{wf} as an $\bigl(N,\vxin\bigr)-$binomial process, where $\vxin:=\bigl(\xi_{_N}(0),\xi_{_N}(1),\ldots,\xi_{_N}(N)\bigr)$ is the vector of conditional frequency expectations with $\xi_{_N}(0)=\xi_{_N}(N)=0$ and $\xi_{_N}(i)\in(0,1)$ for any $i\in\Omega_{_N}^o.$
\begin{theorem}
\label{main3}
Suppose that Assumption~\ref{assume1} is satisfied. Let
\beqn
\label{pimu}
\overrightarrow{\eta_{_N}}(i)=\frac{\gamma^{-1}i}{\gamma^{-1}i+(N-i)}\qquad \mbox{and}\qquad \overrightarrow{\zeta_{_N}}(i)=\frac{\alpha^{-1}i}{\alpha^{-1}i+(N-i)}.
\feqn
Then, in a possibly enlarged probability space, for any $N\geq N_0$ and $i\in\Omega_{_N}^o$ there exists a Markov chain $(\xna_t,\xnb_t,\xnc_t)_{t\in \zz_+}$ on $\Omega_{_N}\times\Omega_{_N}\times\Omega_{_N}$
such that the following holds true:
\begin{enumerate}
\item $(\xna_t)_{t\in \zz_+}$ is an $(N,\overrightarrow{\eta_{_N}})$-binomial process.
\item $(\xnb_t)_{t\in \zz_+}$ is an $\bigl(N,\vxin\bigr)$-binomial process with $\vxin$ given by \eqref{xin}.
\item $(\xnc_t)_{t\in \zz_+}$ is an $(N,\overrightarrow{\zeta_{_N}})$-binomial process.
\item $\xna_0= \xnb_0= \xnc_0=i$ and $\xna_t~\leq~\xnb_t~\leq~\xnc_t$ for all $t\in \nn,$ with probability one.
\end{enumerate}
\end{theorem}
The theorem utilizes in our setting the idea of comparison of a Markov process to a similar but more explicitly understood one (see, for instance,
\cite{illiad,mbounds_queues,monotone5,monotone4,monotone6} and references therein for early work in this direction).
Ignoring the technicalities, the theorem is a particular case of a more general coupling comparison result due to O'Brien \cite{monotone4}
(see in addition, for instance, Theorem~2 in \cite{monotone5} and Theorem~3.1 in \cite{monotone6} for related results). In Section~\ref{proof-main3} we give a
simple self-contained proof of Theorem~\ref{main3}, specifically exploiting a particular structure of binomial processes.
\par
Theorem~\ref{main3} asserts that there exists a coupling such that for almost any realization of the triple process $(\xna_t,\xnb_t,\xnc_t)_{t\in \zz_+}$ the entire trajectory of the frequency-dependent model $(\xnb_t)_{t\in \zz_+},$
whose distribution coincides with the distribution of the chain $(\xn_t)_{t\in \zz_+}$ studied in this paper, is
placed between trajectories of two Wright-Fisher models with constant selection. The result suggests that the qualitative behavior
of $(\xn_t)_{t\in \zz_+}$ in a macroscopic level is similar to those of classical Wright-Fisher models with a constant selection. Furthermore, the hierarchy of the Wright-Fisher models allows to derive lower and upper bounds for important characteristics of our model in terms of the analogous quantities for standard
Wright-Fisher models with selection. We remark that Theorem~\ref{main1} can be deduced from Theorem~\ref{main3} combined with results of Moran in \cite{mbounds} which show that the fixation probabilities of $(\xna_t)_{t\in \zz_+}$ are dominated from below by $\frac{1-\rho^i}{1-\rho^N}$ while the fixation probabilities of $(\xnc_t)_{t\in \zz_+}$ are dominated from above by $\frac{1-\theta^i}{1-\theta^N},$ where $\rho$ and $\theta$ are defined in \eqref{parameters}.
\par
Note both $\eta_{_N}(i)$ and $\zeta_{_N}(i)$ in \eqref{pimu} are non-decreasing functions of $i.$
It turns out that $\xi_{_N}(i)$ has a similar property. More precisely, we have:
\begin{proposition}
\label{propo1}
Let Assumption~\ref{assume1} hold. Then the following holds true for any $N\geq N_0:$
\begin{itemize}
\item [(i)] $\xi_{_N}(i)< \xi_{_N}(i+1)$ for all $i\in\Omega_{_N}^o.$
\item [(ii)] For any $k\in\Omega_{_N}^o$ and $i,j\in\Omega_{_N}^o$ such that $i<j,$ we have
\beq
P\bigl(\xn_{t+1}\leq k\bigl|\xn_t=i\bigr)<P\bigl(\xn_{t+1}\leq k\bigl|\xn_t=j\bigr).
\feq
\end{itemize}
\end{proposition}
The second part of Proposition~\ref{propo1} asserts, using the terminology coined by \cite{illiad}, that $\xn$ is a {\it stochastically monotone} Markov chain.
It is stated without proof in \cite{illiad,mbounds_queues} that $(ii)$ is a direct consequence of $(i)$ (as a matter of fact, it is noted in \cite{illiad,mbounds_queues}
that the binomial model considered in \cite{mbounds} is an example of a stochastically monotone chain). For the reader's convenience we include a short proof of the
implication $(i)\Rightarrow (ii)$ together with the proof of $(i)$ in Section~\ref{propo1-proof}.
\par
With Proposition~\ref{propo1} in hand, we can formally prove the following intuitively obvious statement (see Section~\ref{propo1-proof} for details).
\begin{corollary}
\label{coro1}
Suppose that Assumption~\ref{assume1} is satisfied. Then the following holds true:
\begin{itemize}
\item [(i)]
For a fixed $n\geq N_0$ and $i\in \Omega_{_N}^o,$ consider $p_{_N}(i)$ as a function of the parameters $w$ and $a,b,c,d$
which is defined within the domain described by Assumption~1', namely in
\beq
\cald:=\bigl\{(a,b,c,d,w)\in\RR^5: w\in (0,1),\, a>c>0,~\mbox{\rm and}~b>d>0\bigr\}.
\feq
Then the partial derivatives of $p_{_N}(i)$ with respect to any of the parameters $a,b,c,d,$ and $w$ exist anywhere within $\cald.$ Furthermore,
\beq
&& \frac{\partial p_{_N}(i)}{\partial w}>0,\,~ \frac{\partial p_{_N}(i)}{\partial b}>0,\,~\frac{\partial p_{_N}(i)}{\partial c}<0,\\
&& \frac{\partial p_{_N}(i)}{\partial a}>0~\mbox{\rm unless}~N=2,\,~ \frac{\partial p_{_N}(i)}{\partial d}<0~\mbox{\rm unless}~N=2.
\feq
\item[(ii)] Let Assumption~\ref{assume1} hold. Then for any fixed $N\geq N_0,$ $p_{_N}(i)$ is a strictly increasing function of the parameter
$i$ on $\Omega_{_N}.$
\end{itemize}
\end{corollary}
In other words, $p_{_N}(i)$ is an increasing function of the initial state $i$, and it is also a smooth and strictly monotone function of each of the five parameters $a,b,c,d,$ and $w.$
\subsection{Branching process limit for large populations}
\label{branching}
In this section we consider the asymptotic behavior of the model when the population size approaches infinity.
In view of Theorem~\ref{main1} we have the following bounds for the limiting fixation probability:
\beq
1-\rho^i< \lim_{N\to\infty} p_{_N}(i) <1-\theta^i.
\feq
Thus the following result is a direct implication of Theorem~\ref{main1}.
\begin{corollary}
\label{cor}
Under Assumption~\ref{assume1},
\begin{itemize}
\item[(a)] $\liminf_{N\to\infty} p_{_N}(i)>0,$ $\forall~i\in\nn.$
\item[(b)] $\lim\limits_{N\to\infty} p_{_N}(i_{_N})=1$ for any sequence $i_{_N}\in\Omega_{_N}^o$ such that $\lim\limits_{N\to\infty}i_{_N}=+\infty.$
\end{itemize}
\end{corollary}
The first part of Corollary~\ref{cor} can be refined as follows.
\begin{theorem}
\label{newt}
Let Assumption~\ref{assume1} hold. Then $\lim_{N\to\infty} p_{_N}(i)$ exists and is strictly positive for any $i\in\nn.$
Furthermore,
\beq
\lim_{N\to\infty} p_{_N}(i)=1-q^i,
\feq
where $q$ is the unique in $(0,1)$ root of the equation $q=e^{-\lambda(1-q)}$ with
\beqn
\label{lambda}
\lambda=\frac{1-w+wb}{1-w+wd}.
\feqn
\end{theorem}
Theorem~\ref{newt} implies in particular that just one advantageous mutant can invade an infinite population. A similar result for the frequency-dependent Moran
model of \cite{nature04a,nowakmoran} has been obtained in \cite{antalsm}.
\par
The proof of Theorem~\ref{newt} is based on the approximation of the Wright-Fisher model by a branching process with
a Poisson distribution of offspring. The idea to study the fixation probability of a Wright-Fisher model using a branching process approximation goes back to at least
Fisher \cite{fisher} and Haldane \cite{haldane}. Typically, this approximation scheme is exploited using heuristic or numerical arguments \cite{ewens,gale}.
The proof of Theorem~\ref{newt} given in Section~\ref{proof-newt} is rigorous. A small but essential part of the formal argument is the use of  \textit{a priori} estimates
provided by Theorem~\ref{main1}.
\par
Once it has been established that a single advantageous mutant has a non-zero probability of extinction,
it is natural to ask how long extinction takes, if at all. This question is addressed in the following result.
\begin{theorem}
\label{main7}
Suppose that Assumption~\ref{assume1} is satisfied. Let $\lambda$ and $q$ be as defined in the statement of Theorem~\ref{newt}, and introduce
\beqn
\label{s74}
s_1=\frac{4-\lambda^2q^2}{\lambda q} \qquad \mbox{and} \qquad s_2=\frac{\lambda e^{-\lambda }}{\lambda q+e^{-\lambda q}-1}.
\feqn
Then there exist a constant $C_0>0$ and a function $\theta: (1,\infty)\to (0,\infty)$ that depend only on the payoff matrix $(a,b,c,d)$ and the selection parameter $w,$ such that the following holds true for any real $\eta>1,$ $k,m\in\nn,$ and integers $N\geq N_0,$ $J\in \Omega_{_N}^o:$
\beq
P(T\leq m|\xn_0=k)&\leq& \Bigl(\frac{qs_2(1-\lambda^mq^m)}{s_2-\lambda^mq^m}\Bigr)^k+e^{\theta(\eta) \lambda^{-m}(k\eta^m\lambda^m -N)}
\\
&&
\quad
+m C_0\frac{J^{3/2}}{N}+ e^{\theta(\eta) \lambda^{-m}(k\eta^m\lambda^m -J)}
\feq
and
\beq
P(T\leq m|\xn_0=k)&\geq& \Bigl(\frac{qs_1(1-\lambda^mq^m)}{s_1-\lambda^mq^m}\Bigr)^k
\\
&&
\quad
-m C_0\frac{J^{3/2}}{N}- e^{\theta(\eta) \lambda^{-m}(k\eta^m\lambda^m -J)}.
\feq
\end{theorem}
\begin{remark}
\label{inorder}
A few remarks are in order.
\begin{itemize}
\item[(i)] An explicit upper bound for $C_0$ can be derived from \eqref{dif} and \eqref{diff}.
\item[(ii)] One can set $\theta(\eta)=\min\bigl\{\theta>0:e^x-1\leq \eta x~\mbox{\rm for all}~x\in [0,\theta\lambda^{-1}]\bigr\}.$ This can be seen from the
proof of Lemma~\ref{lemma1} below.
\item[(iii)] The identity $q=e^{-\lambda(1-q)}$ implies $\lambda q<1$ because $e^{-q^{-1}(1-q)}<q$ for any $q\in (0,1)$ and $e^{-\lambda(1-q)}$
is a decreasing function of $\lambda.$ In particular, for $i=1,2,$ we have
\beq
\lim_{m\to\infty} \Bigl(\frac{qs_1(1-\lambda^mq^m)}{s_1-\lambda^mq^m}\Bigr)^k=q^k,
\feq
which is, according to Theorem~\ref{newt} and \eqref{accord}, equivalent to
\beq
1-\lim_{N\to\infty} p_{_N}(k)&=&\lim_{N\to\infty}\lim_{m\to\infty} P(T\leq m, \xn_T=0|\xn_0=k)
\\
&=&
\lim_{m\to\infty}\lim_{N\to\infty} P(T\leq m, \xn_T=0|\xn_0=k).
\feq
On the other hand, a suitable adaptation of the heuristic argument given in Section~6.3.1 of \cite{dgbook} for a Moran model suggests
that
\beq
\lim_{N\to\infty} P(T\leq c\log N, \xn_T=N|\xn_0=k)=0
\feq
as long as $c<C_1$ for some threshold constant $C_1>0.$ If this heuristic is correct then the bounds given in the theorem
are tight for large values of $m$ and $N$ as long as we maintain $m<c \log N$ for some $c<C_1.$
\item[(iv)] The contribution of the correction term $e^{\theta(\eta) \lambda^{-m}(k\eta^m\lambda^m -J)}$ is small for large values of $N$ if,
for instance, one sets $J=N^\alpha$ for some positive real $\alpha<2/3$ and maintain $k\eta^m\lambda^m <cJ$ for some constant $c\in (0,1).$
\end{itemize}
\end{remark}
The proof of Theorem~\ref{main7} is given in Section~\ref{proof-main7}. The main ingredient of the proof is the branching process approximation which
confirms that the first $m$ steps of the Wright-Fisher model look with a high probability like the first $m$ steps of a branching process
with Poisson distribution of offspring. The first steps are the most important ones since
there is little randomness involved in the dynamics of the process for intermediate values of $i,$ where almost deterministically
$\xn_{t+1}=\xi_{_N}\bigl(\xn_t\bigr)>(1+\veps)\xn_t$ for a small $\veps>0$ (Chernoff-Hoeffding bounds for a binomial distribution \cite{Hoeffding} can be used to verify this). Compare also with the three phases of the fixation process described in detail in Section~6.3.1 of \cite{dgbook}.
 To estimate the error of the approximation we use an optimal (so called {\it maximal}) coupling of binomial and Poisson distributions and classical bounds on the total variation distance between the two distributions. Finally, to evaluate the extinction time distribution of the branching process
we use bounds of \cite{agresti74} obtained through the comparison of a Poisson branching process to a branching process with a fractional linear generating function of offspring. We remark that in the context of biological applications, the approximation of an evolutionary process by a branching process with a fractional linear generating function of offspring was apparently first considered in \cite{fractional_bp}.

\subsection{Numerical example}
Consider the following payoff matrix:
\beq
\begin{tabular}{ l | c  r }
     & A & B \\ \hline
    A & 4 & 2 \\
    B & 3 & 1 \\
  \end{tabular}
\feq
Theorem \ref{newt} indicates a very limited influence of the population size $N$ on the fixation probability $p_{_N}(i)$ for large values of $N.$
For illustration purposes we consider a fixed population size $N=100$ and let the selection parameter vary between $0$ and $1.$  Figure \ref{fprob} shows a comparison of numerical and analytical results. The blue line represents the analytically obtained limiting fixation probability $p_\infty:=\lim_{N\to\infty} p_{_N}(1)=1-q$ as a function of the selection parameter $w,$ while the black dots are numerically obtained fixation probabilities of one advantageous mutant for $N=100.$
\begin{figure}[H]
\begin{center}
\includegraphics[height=7cm]{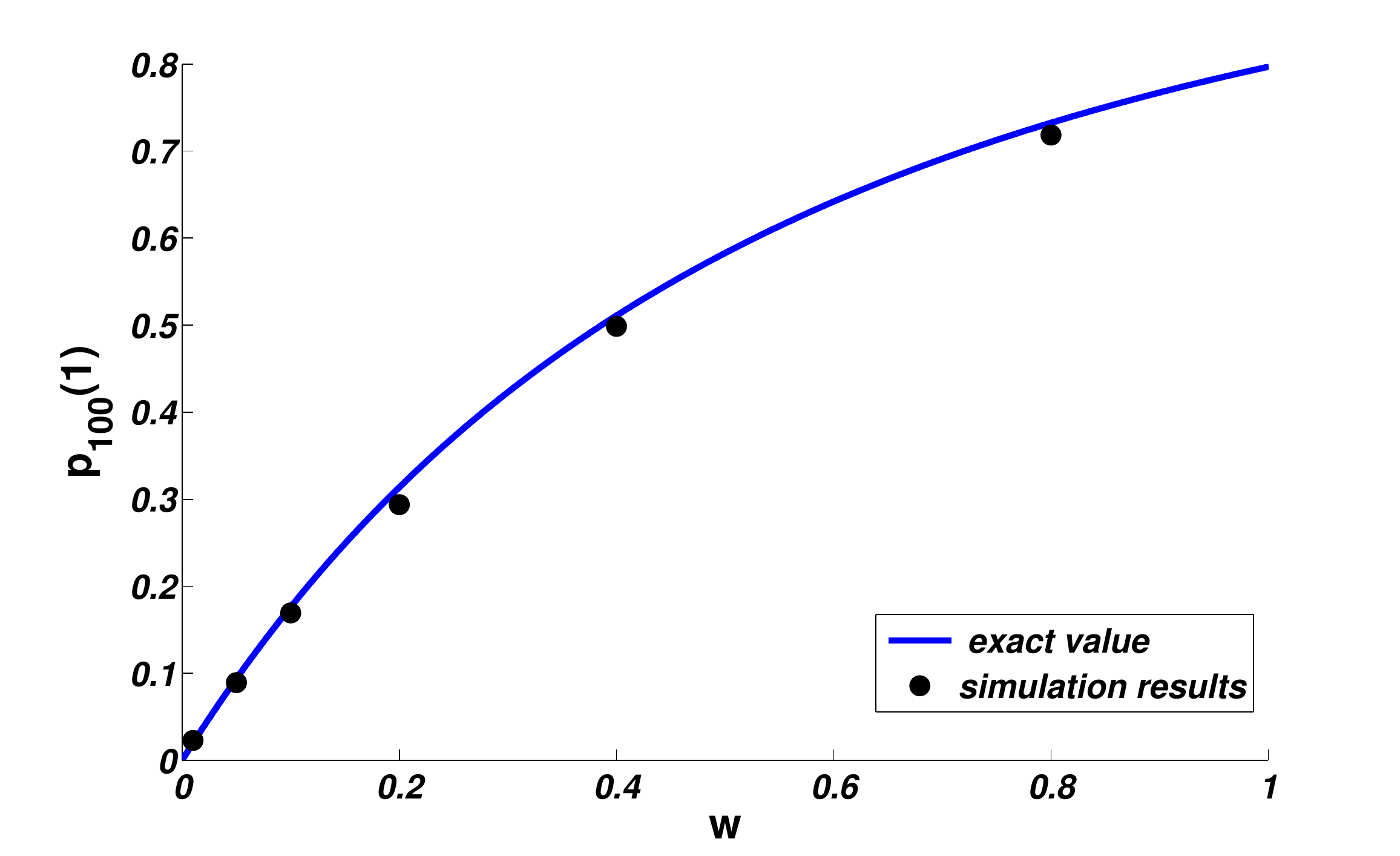}
\end{center}
\caption{The solid blue line represents the fixation probability $p_\infty(1)=1-q$ of a single advantageous mutant in the infinite population limit as a function of the selection parameter $w.$ The black dots represent numerically obtained fixation probabilities $p_{100}(1)$ for six different values of $w.$ Numerical results are obtained by observing $10^4$ realizations of the Markov chain \eqref{wf} with the above specified parameters.}
\label{fprob}
\end{figure}
\par
We also performed numerical simulations for the fixed selection parameter $w=0.3$ and the initial population size varying through $i=1,2,...,10.$ The results of these simulations are shown in Table~\ref{tab1}. In the case of large populations, Theorem \ref{newt} suggests that the fixation probability at zero is given by $q=0.5770.$ We numerically obtained the fixation probabilities $p_{_N}(i)$ for the above specified parameters and used a nonlinear least squares routine in MATLAB to find the best fitting $q_{_N}$ assuming that $p_{_N}(i)=1-q_{_N}^i.$ Table \ref{tab1} shows the results of this nonlinear fitting $q_{_N}$ and the differences $q_{_N}-q$ for the specified values of the population sizes $N.$
\begin{figure}[H]
\centering
\includegraphics[height=12cm]{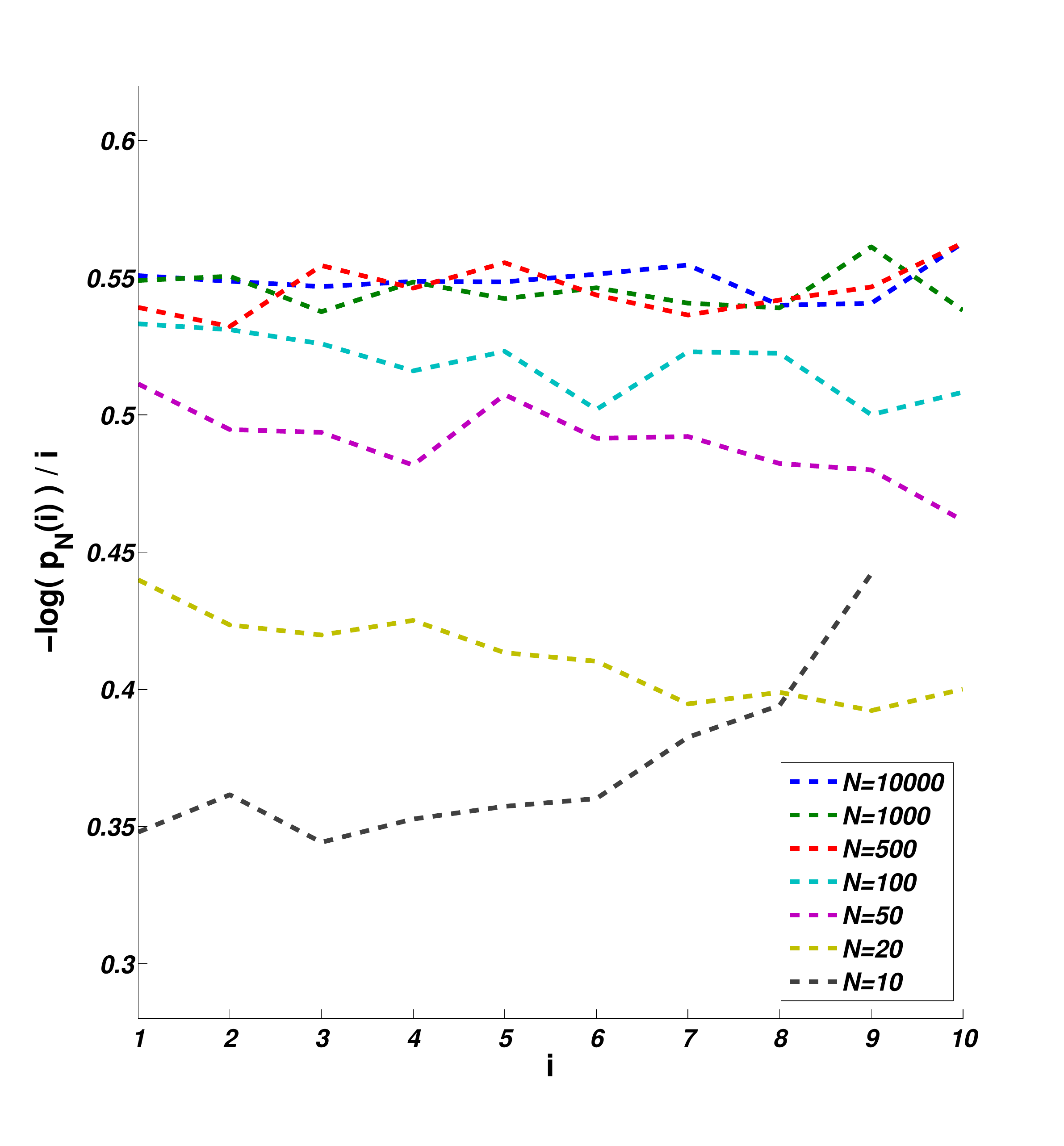}
\caption{Plot of the function $-\frac{1}{i}\log p_{_N}(i)$ determined in simulations for
various values of $N$ and several small values of the initial state $i.$}
\label{logplot}
\end{figure}

\begin{table}[h!]
\begin{center}
\begin{tabular}{llllllll}
$N$     & 10     & 20     & 50     & 100    & 500    & 1000   & 10000  \\
$q_{_N}$   & 0.6979 & 0.6567 & 0.6090 & 0.5909 & 0.5812 & 0.5792 & 0.5776 \\
$q_{_N}-q$ & 0.1209 & 0.0797 & 0.0320 & 0.0139 & 0.0042 & 0.0022 & 0.0006\\
\end{tabular}\\
\end{center}
\caption{$q_{_N}$ is the value obtained from the nonlinear least squares fitting of numerically obtained fixation probabilities starting with $i=1,2,...,10$ individuals. Up to N=1000, we realized the Markov chain \eqref{wf} $10^4$ times and for $N=10000$ we used $5\times 10^4$ realizations.}
\label{tab1}
\end{table}
\section{Conclusion}
\label{conclusion}
In this paper, we considered the fixation probability of symmetric games in Wright-Fisher processes with directional selection describing evolutionary dynamics of two types. Our analysis shows the existence of exponential lower and upper bounds for the fixation probabilities for any population size $N\in\nn.$ Using these facts one can draw the following biological conclusions.
\begin{enumerate}

 \item The fixation probabilities of an advantageous or a deleterious mutant in a population of size $N$ depend on both population size and the relative fitnesses of the  phenotypes.

\item In the case of advantageous mutants, the dependence on the population size is weak, i.e. the lower bound on the fixation probability is bounded below by a positive constant depending only on the fitness of the two phenotypes.

\item The fixation probability $q_{_N}(i)$ of $i$ deleterious mutants is an exponentially decreasing function of $N.$

\end{enumerate}

In addition, we studied the asymptotics of the fixation probability as the population size goes to infinity.
We showed that
\begin{itemize}
\item[4.] A single advantageous mutant can invade an infinite population with positive probability.

\item[5.] Whenever the initial population $\xn_0$ of advantageous players is unbounded as $N$ goes to infinity (even if its proportion $\xn_0/N$ vanishes to zero asymptotically), the fixation probability $p_{_N}(i)$ is asymptotically zero for deleterious players and one for advantageous players.

\end{itemize}

\section{Proofs}
\label{proofs}
This section, divided into five subsections, contains proofs of the results stated in Section~\ref{results}.
The proof of Lemma~\ref{lem2} is given in Section~\ref{proof-lem2}. The proof of Theorem~\ref{main3}
is included in Section~\ref{proof-main3}. Section~\ref{proof-newt} is devoted to the proof of Theorem~\ref{newt}.
Finally, the proof of Theorem~\ref{main7} is given in Section~\ref{proof-main7}.
\subsection{Proof of Lemma \ref{lem2}}
\label{proof-lem2}
First, observe that if $X$ is a binomial random variable $BIN(N,\xi),$ then for any constant $\rho\in\RR,$
\beqn
\label{gf-bin}
E(\rho^X)&=&\sum_{k=0}^N \rho^k{N \choose k}\xi^k(1-\xi)^{N-k}=(\xi\rho+1-\xi)^N.
\feqn
Thus, in order to prove part (a) of the lemma we need to show that the following inequality
holds for some $\rho\in(0,1)$ and all $N\geq N_0$ and $i\in\Omega_{_N}^o:$
\beq
\bigl(\rho\xi_{_N}(i)+1-\xi_{_N}(i)\bigr)^N\leq \rho^i.
\feq
Using the notation $x=i/N,$ the above inequality can be rewritten as
\beq
1-(1-\rho)\xi_{_N}(Nx)\leq \rho^x .
\feq
It follows from \eqref{xin} and Assumption~\ref{assume1} that $\xi_{_N}(Nx)\geq \frac{x}{x+(1-x)\gamma}.$
Thus, it suffices to show that for some constant $\rho\in(0,1)$ and all $x\in (0,1),$
\beqn
\label{mref}
\frac{x}{x+(1-x)\gamma} -\frac{1-\rho^x}{1-\rho}\geq 0.
\feqn
$\mbox{}$
\\
Similarly, in order to prove part (b) of the lemma it is sufficient to show that for some constant $\theta\in(0,1)$ and all $x\in (0,1),$
\beqn
\label{gin}
\frac{x}{x+(1-x)\alpha} -\frac{1-\theta^x}{1-\theta}\leq 0.
\feqn
Inequalities \eqref{mref} and \eqref{gin} have been analyzed in a similar context by Moran \cite{mbounds} (see specifically the bottom of p.~488 in \cite{mbounds})
who found the feasible solutions given in \eqref{parameters}. The proof of the lemma is complete.\qed
\subsection{Proof of Theorem~\ref{main3}}
\label{proof-main3}
The proof relies on a standard coupling argument. Fix any $N\geq N_0.$ It follows from Assumption~\ref{assume1} that
\beqn
\label{mono}
\eta_{_N}(i)\leq \xi_{_N}(i)\leq \zeta_{_N}(i),\quad \forall~i\in\Omega_{_N}^o.
\feqn
Let $\bigl(U_{t,k}\bigr)_{t\in \zz_+,k\in\nn}$ be a sequence of independent random variables, each one distributed
uniformly on the interval $(0,1).$ Using the interpretation of the binomial random variable as a superposition of independent
Bernoulli trials, the Markov chain $\bigl(\xna_t,\xnb_t,\xnc_t\bigr)_{t\in \zz_+}$ can be constructed inductively in the following manner.
For each $t\in \zz_+,$ given \\ $\bigl(\xna_t,\xnb_t,\xnc_t\bigr),$
define Bernoulli random variables $\bigl(b^{(1)}_{t,k},b^{(2)}_{t,k},b^{(3)}_{t,k} \bigr)_{1\leq k\leq N}$ as follows:
\beq
b^{(t)}_{k,1}&=&
\left\{
\begin{array}{lll}
1&\mbox{\rm if}&U_{t,k}\leq \eta_{_N}(\xna_t) \\
0&\mbox{\rm if}&U_{t,k}> \eta_{_N}(\xna_t),
\end{array}
\right.
\\
b^{(t)}_{k,2}&=&
\left\{
\begin{array}{lll}
1&\mbox{\rm if}&U_{t,k}\leq \xi_{_N}(\xnb_t) \\
0&\mbox{\rm if}&U_{t,k}> \xi_{_N}(\xnb_t),
\end{array}
\right.
\\
b^{(t)}_{k,3}&=&
\left\{
\begin{array}{lll}
1&\mbox{\rm if}&U_{t,k}\leq \zeta_{_N}(\xnc_t) \\
0&\mbox{\rm if}&U_{t,k}> \zeta_{_N}(\xnc_t),
\end{array}
\right.
\feq
and set
\beqn
\label{abc}
\xna_{t+1}=\sum_{k=1}^N b^{(t)}_{k,1},\quad \xnb_{t+1}=\sum_{k=1}^N b^{(t)}_{k,2},\quad \xnc_{t+1}=\sum_{k=1}^N b^{(t)}_{k,3}.
\feqn
It follows from \eqref{mono} and the fact that both $\eta_{_N}(i)$ and $\zeta_{_N}(i)$ are monotone increasing functions of $i,$ that the inequality
$\xna_t\leq \xnb_t\leq \xnc_t$ implies that $\xi_{_N}(\xnb_t)\geq \eta_{_N}(\xnb_t)\geq$ $\eta_{_N}(\xna_t)$ and $\xi_{_N}(\xnb_t)\leq \zeta_{_N}(\xnb_t)\leq \zeta_{_N}(\xnc_t),$ and hence $\eta_{_N}(\xna_t)\leq \xi_{_N}(\xnb_t)\leq$ $\zeta_{_N}(\xnc_t).$ By virtue of \eqref{abc},
the latter inequalities along with $\xna_t\leq \xnb_t\leq \xnc_t$ imply $\xna_{t+1}\leq \xnb_{t+1}\leq \xnc_{t+1},$ and the claim follows by induction on $t.$
\qed
\subsection{Proofs of Proposition~\ref{propo1} and Corollary~\ref{coro1}}
\label{propo1-proof}
We start with the proof of the proposition.
\begin{proof}[Proof of Proposition~\ref{propo1}] $\mbox{}$ \\
{\bf (i)} First, observe that
\beq
\xi_{_N}(i)=\frac{i}{i+(N-i)\frac{g_{_N}(i)}{f_{_N}(i)}}<\xi_{_N}(i+1)=\frac{i+1}{i+1+(N-i-1)\frac{g_{_N}(i+1)}{f_{_N}(i+1)}}
\feq
if and only if
\beqn
\label{eval}
\frac{g_{_N}(i+1)}{g_{_N}(i)}\cdot \frac{f_{_N}(i)}{f_{_N}(i+1)}<\frac{i+1}{i}\cdot \frac{N-i}{N-i-1}.
\feqn
To evaluate the left-hand side of the inequality in \eqref{eval} we will use the following simple fact.
\begin{lemma}
\label{maxi}
For any strictly positive reals $A,B,C,D,$
\beq
\frac{A+B}{C+D}\leq \max\Bigl\{\frac{A}{C},\frac{B}{D}\Bigr\}.
\feq
Furthermore, the equality holds if and only if $\frac{A}{C}=\frac{B}{D}.$
\end{lemma}
It follows from the lemma that
\beqn
\label{chain}
\begin{split}
&
\frac{g_{_N}(i+1)}{g_{_N}(i)}\cdot \frac{f_{_N}(i)}{f_{_N}(i+1)}\leq
\\
&
\qquad
\leq \max\Bigl\{1,\frac{\pi_{_B}(i+1,N)}{\pi_{_B}(i,N)}\Bigr\}\cdot \max\Bigl\{1,\frac{\pi_{_A}(i,N)}{\pi_{_A}(i+1,N)}\Bigr\}
\\
&
\qquad
\leq \max\Bigl\{1,\frac{i+1}{i},\frac{N-i-2}{N-i-1}\Bigr\}\cdot \max\Bigl\{1,\frac{i-1}{i},\frac{N-i}{N-i-1}\Bigr\}
\\
&
\qquad
\leq \frac{i+1}{i}\cdot \frac{N-i}{N-i-1}.
\end{split}
\feqn
Thus
\beqn
\label{strict}
\frac{g_{_N}(i+1)}{g_{_N}(i)}\cdot \frac{f_{_N}(i)}{f_{_N}(i+1)}\leq \frac{i+1}{i}\cdot \frac{N-i}{N-i-1}
\feqn
Furthermore, the equality is only possible if the equality holds everywhere in the chain of inequalities \eqref{chain}.
However, by the lemma, if the equality holds everywhere in \eqref{chain}, then, in particular,
\beq
1=\frac{\pi_{_B}(i+1,N)}{\pi_{_B}(i,N)}=\frac{\pi_{_A}(i,N)}{\pi_{_A}(i+1,N)},
\feq
in which case $\frac{g_{_N}(i+1)}{g_{_N}(i)}\cdot \frac{f_{_N}(i)}{f_{_N}(i+1)}=1.$ The contradiction shows that the inequality
in \eqref{strict} is strict, which completes the proof of part (i) of the proposition.
\\
$\mbox{}$
\\
{\bf (ii)} We remark in passing that part (ii) of the proposition can be proved using a coupling of the conditional distributions
$P\bigl(\xn_{t+1}\leq \,\cdot\,\bigl|\xn_t=i\bigr)$ and $P\bigl(\xn_{t+1}\leq \,\cdot\,\bigl|\xn_t=j\bigr),$ similar to the one which we have exploited
in the proof of Theorem~\ref{main3}. Alternatively, and more in the spirit of \cite{mbounds}, in view of the result in part (i), in order to establish
the claim in part (ii) it suffices to verify that $\frac{d}{dx} T_{k,N}(x)<0$ for
$x\in (0,1),$ where $T_{k,N}(x):=\sum_{i=0}^{k}{N \choose i} x^i(1-x)^{N-i}.$ To this end, write
\beq
&&
\frac{d}{dx} T_{k,N}(x)
\\
&&
\quad
=\sum_{i=1}^{k}{N \choose i} ix^{i-1}(1-x)^{N-i}-\sum_{i=0}^k{N \choose i} (N-i)x^i(1-x)^{N-i-1}
\\
&&
\quad
=N\sum_{i=1}^{k}{N-1 \choose i-1} x^{i-1}(1-x)^{N-i}-N\sum_{i=0}^k{N-1 \choose i} x^i(1-x)^{N-i-1}.
\feq
Changing variable $i$ to $j=i-1$ in the first sum that appears in the last line above, we obtain
\beq
&&
\frac{d}{dx} T_{k,N}(x)
\\
&&
\quad
=N\sum_{j=0}^{k-1}{N-1 \choose j} x^j(1-x)^{N-j-1}-N\sum_{i=0}^k{N-1 \choose i} x^i(1-x)^{N-i-1}
\\
&&
=-N{N-1 \choose k} x^k(1-x)^{N-k-1}<0.
\feq
The proof of the proposition is complete.
\end{proof}
We proceed with the proof of the corollary.
\begin{proof}[Proof of Corollary~\ref{coro1}]
For $i,j\in\Omega_{_N}$ let $\qn_{i,j}=P\bigl(\xn_{t+1}=j\bigl|\xn_t=i\bigr).$ Then, by the Markov property,
\beq
p_{_N}(i)=\sum_{j=1}^{N-1}\qn_{i,j}p_{_N}(j)+\qn_{i,N},\qquad \forall~i\in\Omega_{_N}^o.
\feq
Given $N$ and $i,j,$ consider $\qn_{i,j}$ as a function $\RR^5\to \RR$ of the five independent variables $a,b,c,d,$ and $w.$  The existence of the partial derivatives of $p_{_N}(i)$ with respect to these variables follows from the implicit function theorem applied to the function $f=(f_1,\ldots,f_{N-1}):$ $\RR^{5+N-1}\to \RR^{N-1},$ where
\beq
f_i\bigl(a,b,c,d,w,p_{_N}(1),\ldots,p_{_N}(N-1)\bigr):=p_{_N}(i)-\sum_{j=1}^{N-1}\qn_{i,j}p_{_N}(j)-\qn_{i,N}.
\feq
The monotonicity of $p_{_N}(i)$ on each of the parameters $a,b,c,d,w,$ and $i$ follows then from the corresponding monotonicity
of $\xi_{_N}(i)$ and the following version of O'Brien's results in \cite{monotone4}:
\begin{proposition}
\label{main3l}
Let $\overrightarrow{\eta_{_N}}=(\eta_i)_{0=1}^N$ and $\vxin=(\xi_i)_{i=0}^N$ be two vectors in $\RR^{N+1}$ such that
\begin{itemize}
\item [(i)] $\eta_0=\xi_0=0,$ $\eta_N=\xi_N=1,$ and $0<\eta_i\leq \xi_i<1$ for any $i\in\Omega_{_N}^o.$
\item [(ii)] Either $\eta_i\leq \eta_{i+1}$ for all $i\in\Omega_{_N}^o$ or $\xi_i\leq \xi_{i+1}$ for all $i\in\Omega_{_N}^o.$
\end{itemize}
Then for any $i,j\in\Omega_{_N}^o,$ $i\leq j,$ there exists a Markov chain $(\xna_t,\xnb_t)_{t\in \zz_+}$ on $\Omega_{_N}\times\Omega_{_N}$
with the following properties:
\begin{enumerate}
\item $(\xna_t)_{t\in \zz_+}$ is an $(N,\overrightarrow{\eta_{_N}})$-binomial process.
\item $(\xnb_t)_{t\in \zz_+}$ is an $\bigl(N,\vxin\bigr)$-binomial process.
\item With probability one, $\xna_0=i,$ $\xnb_0=j,$ and $\xna_t\leq \xnb_t$ for $t\in \zz_+.$
\end{enumerate}
\end{proposition}
We remark that Theorem~\ref{main3} and Proposition~\ref{main3l} are two variants of the same result, and a self-contained proof of the latter can be obtained
using a coupling argument similar to the one we employed in the proof of the theorem.
\end{proof}
\subsection{Proof of Theorem~\ref{newt}}
\label{proof-newt}
Throughout the argument we formally treat the process $\xn$ as a Markov chain on $\zz_+=\nn\cup\{0\}$ with absorbtion states at $0$ and $N,N+1,N+2,\ldots,$
and assume that all chains $\xn,$ $N\geq N_0,$ have a common initial state, a given integer $i_0\in\nn.$
\par
First, observe that for any fixed $i\in\nn,$
\beq
\lim_{N\to\infty}\xi_{_N}(i)N=\lambda i,
\feq
where $\lambda$ is defined in \eqref{lambda}. Therefore, for any fixed pair of integer states $i>0$ and $j\geq 0,$ and an integer time $t\in \zz_+,$
\beqn
\label{kernelc}
\lim_{N\to\infty} P\bigl(\xn_{t+1}=j\bigl|\xn_t=i\bigr)=e^{-\lambda i}\frac{(\lambda i)^j}{j!}.
\feqn
Let $Z=\bigl(Z_t\bigr)_{t\in \zz_+}$ be a Markov chain on $\zz_+$ with absorption state at zero and Poisson transition kernels
\beq
P\bigl(Z_{t+1}=j\bigl|Z_t=i\bigr)=e^{-\lambda i}\frac{(\lambda i)^j}{j!},\qquad i\in \nn~\mbox{and}~j\geq 0.
\feq
Assume that the Markov chain $Z$ has the same initial state $Z_0=i_0$ as any $\xn,$ $N\geq N_0.$
Since the sum of two independent Poisson random variables is a Poisson random variable with the parameter equal to the sum of their parameters,
we can assume without loss of generality that $Z$ is a Galton-Watson branching process with a Poisson offspring distribution.
More precisely, we assume that (cf. \cite[Section~1.4]{ewens})
\beqn
\label{br}
Z_{t+1}=\sum_{k=1}^{Z_t} Y_{t,k}
\feqn
for some independent random variables $Y_{t,k},$ $t\in \zz_+,k\in\nn,$ each one distributed as Poisson($\lambda$), namely
\beq
P(Y_{t,k}=j)=e^{-\lambda}\frac{\lambda^j}{j!}, \quad j\geq 0,
\feq
for all $t\in \zz_+$ and $k\in\nn,$ with the parameter $\lambda$ introduced in \eqref{lambda}. As usual, we convene that the sum in \eqref{br} is empty if $Z_t=0.$
\par
The convergence of the transition kernels in \eqref{kernelc} implies the weak convergence
of the sequence of Markov chains  $\xn$ to the branching process $Z$ (see, for instance, Theorem~1 in \cite{karr}). Since $\lambda>1$ under the conditions
of Theorem~\ref{newt} (recall that $w>0$ and $b>d$), it follows that
\beqn
\label{ztl}
P\bigl(\lim_{t\to\infty} Z_t=+\infty\bigr)>0.
\feqn
Let $\tn_0=\inf\bigl\{t>0:\xn_t=0\bigr\}$ and $T_0=\inf\{t>0:Z_t=0\}$
be the first hitting time of zero by the Markov chains $\xn$ and $Z,$ respectively.
Convene, as usual, that the infimum of an empty set is $+\infty.$
$T_0$ is the extinction time of the branching process $Z,$ and in this notation \eqref{ztl} reads $P(T_0<\infty)<1.$
In fact (see, for instance, \cite[Section~1.4]{ewens}), $q=\bigl[P(T_0<\infty)\bigr]^{1/i_0}$ is the unique in $(0,1)$ root of the fixed-point equation $q=e^{-\lambda(1-q)},$
whose right-hand side is the moment-generating function of $Y_{t,k}$ evaluated at $q\in (0,1).$
\par
Since the transition kernel of $(\xn_t)_{t\in \zz_+}$ converges to that of $(Z_t)_{t\in \zz_+},$
\beq
\lim_{N\to\infty} P\bigl(\tn_0<K\bigr)=P(T_0<K),\qquad \forall~K\in\nn.
\feq
Therefore,
\beqn
\label{step}
\lim_{K\to\infty}\lim_{N\to\infty} P\bigl(\tn_0<K\bigr)=P(T_0<\infty).
\feqn
We will conclude the proof of Theorem~\ref{newt} by showing that we can interchange the limits
in the above identity, and hence
\beqn
\label{accord}
\lim_{N\to\infty} P\bigl(\tn_0<\infty\bigr)=\lim_{N\to\infty}\lim_{K\to\infty} P\bigl(\tn_0<K\bigr)=P(T_0<\infty).
\feqn
To this end, write
\begin{align}
\label{ai}
\nonumber
&\bigl|P\bigl(\tn_0<\infty\bigr)-P(T_0<\infty)\bigr|
\leq
\bigl|P\bigl(\tn_0<\infty\bigr)-P\bigl(\tn_0<K\bigr)\bigr|
\\
\nonumber
&
\qquad
\qquad
+
\bigl|P\bigl(\tn_0<K\bigr)-P(T_0<K)\bigr|
+
\bigl|P(T_0<K)-P(T_0<\infty)\bigr|
\\
&
\qquad
\qquad
:=
A_1(N,K)+A_2(N,K)+A_3(N,K),
\end{align}
where the last line serves to define the events $A_i(N,K),$ $i=1,2,3.$
\par
Pick any $\veps>0.$ First we will estimate
\beq
A_1(N,K)=\bigl|P\bigl(\tn_0<\infty\bigr)-P(\tn_0<K)\bigr|=P\bigl(K\leq \tn_0<\infty\bigr).
\feq
It follows from Assumption~\ref{assume1} that for all $i,N\in\nn$ we have
\beq
\xi_{_N}(i)=\frac{i}{i+(N-i)g_{_N}(i)/f_{_N}(i)}\leq \frac{i}{\alpha N}.
\feq
Therefore, using the strong Markov property and the lower bound in Theorem~\ref{main1}, we obtain for any $m\in\nn$
and $N$ sufficiently large,
\beqn
\label{compom}
\begin{split}
A_1(N,K)&=P\bigl(K\leq \tn_0<\infty\bigr)
\\
&
=P\Bigl(K\leq \tn_0<\infty,\max_{0\leq t \leq K-1}\xn_t<m\Bigr)
\\
&
\qquad
+P\Bigl(K\leq \tn_0<\infty,\max_{0\leq t \leq K-1}\xn_t\geq m\Bigr)
\\
&
\leq
P\Bigl(\xn_1\neq 0,\ldots, \xn_{K-1}\neq 0,\,\max_{0\leq t \leq K-1}\xn_t<m\Bigr)
\\
&
\qquad
+\sum_{j=m}^{N-1} P\Bigl(\max_{0\leq t \leq K-1}\xn_t=j\Bigr)\cdot \bigl(1-p_{_N}(j)\bigr)
\\
&\leq
\bigl[1-P\bigl(\xn_{t+1}=0\bigl|\xn_t=m\bigr)\Bigr]^{K-1}+\bigl(1-p_{_N}(m)\bigr)
\\
&\leq
\Bigl[1-\Bigl(1-\frac{m}{\alpha N}\Bigr)^N\Bigr]^{K-1}+\frac{\rho^m}{1-\rho^N}.
\end{split}
\feqn
Choose now $m_0\in\nn$ so large that
\beq
\limsup_{N\to\infty}\frac{\rho^{m_0}}{1-\rho^N}=\rho^{m_0}\leq \frac{\veps}{6},
\feq
and then $K_1$ so large that for any $K>K_1,$
\beq
\limsup_{N\to\infty} \Bigl[1-\Bigl(1-\frac{m_0}{\alpha N}\Bigr)^N\Bigr]^K=
\Bigl[1-\exp\Bigl({-\frac{m_0}{\alpha}}\Bigr)\Bigr]^K\leq \frac{\veps}{6}.
\feq
Then, for any $K>K_1,$
\beq
\limsup_{N\to\infty} A_1(N,K)\leq \frac{\veps}{3}.
\feq
Find now $K_2\in\nn$ such that for any $K>K_2$ we have
\beq
A_3(N,K)=\bigl|P(T_0<K)-P(T_0<\infty)\bigr|\leq \frac{\veps}{3}.
\feq
Finally, pick any $K_0>\max\{K_1,K_2\},$ and then using \eqref{step} choose $N_1$ such that $N>N_1$ implies
\beq
A_2(N,K_0)=\bigl|P\bigl(\tn_0<K_0\bigr)-P(T_0<K_0)\bigr|\leq \frac{\veps}{3}.
\feq
It follows from the above estimates for $A_i(N,K_0),$ $i=1,2,3,$ and the basic inequality \eqref{ai} that
\beq
\limsup_{N\to\infty} \bigl|P\bigl(\tn_0<\infty\bigr)-P(T_0<\infty)\bigr|\leq \veps.
\feq
Since $\veps>0$ is arbitrary positive number,
\beq
\lim_{N\to\infty} \bigl|P\bigl(\tn_0<\infty\bigr)-P(T_0<\infty)\bigr|=0.
\feq
This establishes \eqref{accord}, and therefore completes the proof of Theorem~\ref{newt}.\qed
\subsection{Proof of Theorem~\ref{main7}}
\label{proof-main7}
The proof of the theorem is broken up into a series of lemmas. Throughout the argument we continue to use notations introduced in Section~\ref{proof-newt}.
We will use a certain {\it optimal coupling} between the branching process $(Z_t)_{t\in \zz_+}$ and the Wright-Fisher model $(\xn_t)_{t\in \zz_+}.$ By coupling we mean constructing in the same probability space a joint distribution of a pair of processes such that their marginal distributions coincide with those of $(Z_t)_{t\in \zz_+}$ and $(\xn_t)_{t\in \zz_+}.$ With a slight abuse of notation we will denote by $(\xn_t,Z_t)_{t\in \zz_+}$ the process of pairs constructed below, thus preserving the original names for the two marginal components of the coupled process. The construction specifically aims to minimize (and also enable us to estimate) $P(\xn_t\neq Z_t).$ We fix $k\in\nn$ and assume throughout the proof that $N>k$ and $\xn_0=Z_0=k.$
\par
To explain the coupling construction, we need to recall the following general result (see, for instance, Appendix~A1 in \cite{barbour}):
\begin{proposition}
\label{coupling}
Let $X$ and $Y$ be two random variables and $\delta(X,Y)\geq 0$ be the total variation distance between their distributions. That is,
\beq
\delta(X,Y):=\sup_A |P(X\in A)-P(Y\in A)|,
\feq
where the supremum is taken over measurable subsets $A$ of the real line. Then there exists a probability space
and a random pair $\bigl(\witi X,\witi Y\bigr)$ defined on the same probability space such that
\begin{enumerate}
\item $\witi X$ is distributed the same as $X.$
\item $\witi Y$ is distributed the same as $Y.$
\item $P\bigl(\witi X\neq \witi Y)=\delta(X,Y).$
\end{enumerate}
\end{proposition}
If $X$ and $Y$ are defined on the set of non-negative integers, the total variation distance $\delta(X,Y)$ is equal to
$\frac{1}{2}\sum_{n=0}^\infty |P(X=n)-P(Y=n)|.$ The coupling described in Proposition~\ref{coupling} is often called a {\it maximal coupling}
of random variables $X$ and $Y.$
\par
We will also use the following coupling inequalities from the literature (for the first assertion see, for instance, Theorem~4 in \cite{poisson} and for the second one Theorem~1.C(i) in \cite{barbour}):
\begin{proposition}
\label{lemma-approx}
\item[(i)] Let $X=BIN(N,p)$ be a binomial random variable with parameters $N\in\nn$ and $p\in[0,1])$ and $Y=Poisson(Np)$ be
a Poisson random variable with parameter $\lambda=Np.$ If $\lambda>1,$ then $\delta(X,Y)\leq \frac{p}{2}.$
\item[(ii)] Let $X$ and $Y$ be two Poisson random variables with parameters $\lambda>1$ and $\mu>0,$ respectively.
Then $\delta(X,Y)\leq \frac{1}{\sqrt{\lambda}}|\mu-\lambda|.$
\end{proposition}
Using the above results, we can construct a coupling of the Wright-Fisher Markov chain $(\xn_t)_{t\in \zz_+}$ and the branching process $(Z_t)_{t\in \zz_+}$ as follows.
The resulting joint process $(\xn_t,Z_t)_{t\in \zz_+}$ will be a Markov chain, and we are now in position to describe its transition kernel.
Suppose that the random pairs $(\xn_s,Z_s)$ have been defined and sampled for all $s\leq t$ and that $\xn_s=Z_s$ for all $s\leq t.$
Let $i_t$ be the common value of $Z_t$ and $\xn_t.$ Then, using the above results and at first approximating $\xn_{t+1}$ by a
 Poisson random variable with parameter $N\xi_{_N}(i_t),$ we can construct the pair $(\xn_{t+1},Z_{t+1})$
in such a way that
\beqn
\nonumber
\label{dif}
&&
P(\xn_{t+1}\neq Z_{t+1}|X_t=Z_t=i_t)
\\
&&
\qquad
\qquad
\leq \frac{1}{2}\xi_{_N}(i_t)+\frac{1}{\sqrt{\lambda i_t}}|N\xi_{_N}(i_t)-\lambda i_t|
\\
\nonumber
&&
\qquad
\qquad
\leq \frac{i_t}{2\alpha N}+\frac{1}{\sqrt{\lambda i_t}}\Bigl|N\xi_{_N}(i_t)-\frac{f_{_N}(i_t)}{g_{_N}(i_t)}i_t\Bigr|+\frac{1}{\sqrt{\lambda i_t}}\Bigl|\frac{f_{_N}(i_t)}{g_{_N}(i_t)}i_t-\lambda i_t\Bigr|.
\feqn
A bit tedious but straightforward calculation shows that in this coupling construction
\beqn
\label{diff}
P(\xn_{t+1}\neq Z_{t+1}|X_t=Z_t=i_t)\leq C_0\frac{i_t^{3/2}}{N},
\feqn
where $C_0$ is a constant which depends only on the payoff matrix $(a,b,c,d)$ and the selection parameter $w,$ but is independent of $i_t$ and $N.$
\par
Once $\xn_t\neq Z_t$ occurs for the first time, we can continue to run the processes $(\xn_s)_{s\geq t}$ and $(Z_s)_{s\geq t}$ independently of each other.
The above discussion is summarized in the following lemma.
\begin{lemma}
\label{pairs}
There exist a probability space and a constant $C_0>0$ which depends on the payoff matrix $(a,b,c,d)$ only,
such that the processes $(\xn_t)_{t\in \zz_+}$ and $(Z_t)_{t\in \zz_+}$ can be defined jointly in this probability space
and the following holds true:
\begin{enumerate}
\item The pairs $(\xn_t,Z_t)$ form a Markov chain.
\item The inequality  in \eqref{diff} is satisfied for any $i_t\in \Omega_{_N}.$
\item For any $i,j,k,m\in \Omega_{_N}$ such that $i\neq j,$ we have
\beq
&&
P(\xn_{t+1}=k, Z_{t+1}=m|X_t=i,Z_t=j)
\\
&&
\qquad
=P(\xn_{t+1}=k|X_t=i)\cdot P(Z_{t+1}=m|Z_t=j).
\feq
\end{enumerate}
\end{lemma}
In the rest of the proof of Theorem~\ref{main7} we will consider the Markov chain $(\xn_t,Z_t)_{t\in \zz_+}$ as described in Lemma~\ref{pairs}.
For $k\in\nn$ we will denote by $P_k$ the distribution of the Markov chain $(\xn_t,Z_t)_{t\in\zz_+}$ conditioned on $\xn_0=Z_0=k.$
We will denote by $E_k$ the corresponding expectation operator.
\par
Let $\tau_{_N}>0$ be the first time when the path of the Wright-Fisher model diverges from the
path of the branching process, that is
\beqn
\label{taun}
\tau_{_N}=\inf\{t\in\nn: \xn_t\neq Z_t\}.
\feqn
In the above coupling construction, as long as $Z_t=\xn_t$ the next pair $(\xn_{t+1},Z_{t+1})$ is sampled using the maximal coupling,
and after the first time when it occurs that $Z_t\neq \xn_t$ we continue to sample $(Z_s)_{s\geq \tau_{_N}}$ and $(\xn_s)_{s\geq \tau_{_N}}$ independently. We remark that the third property in the conclusion of Lemma~\ref{pairs}
(eventual independence of the marginal processes) will never be used in our proof and is needed only to formally specify in a certain way the construction of the coupled Markov chain for all times $t\in \zz_+.$
In fact, we are going to observe and study the properties of the coupled chain only up to the time $\tau_{_N}.$
\par
Fix now any $m\in\nn$ and $\eta>1.$ We will consider only large enough values of $N,$ namely we will assume throughout that $N>(\lambda\eta)^m.$
Recall $T$ from \eqref{ti}. Similarly, for the branching process $Z_t$ define
\beq
\sigma_{_N}=\inf\{t\in \nn: Z_t=0~\mbox{or}~Z_t\geq N\}.
\feq
Recall $\tau_{_N}$ from \eqref{taun}. To evaluate the distribution function of $T$ we will use the following basic inequalities valid for any $k,m\in\nn:$
\beqn
\label{basic}
\nonumber
P_k(T\leq m)&\leq& P_k(T\leq m,\tau_{_N}>m)+P_k(\tau_{_N}\leq m)
\\
&\leq&
 P_k(\sigma_{N}\leq m)+P_k(\tau_{_N}\leq m)
\feqn
and
\beqn
\label{basic1}
\nonumber
P_k(T\leq m)&\geq& P_k(T\leq m,\tau_{_N}>m)
\\
&\geq&
 P_k(\sigma_{N}\leq m)-P_k(\tau_{_N}\leq m).
\feqn
In the next two lemmas we estimate $P(\sigma_{_N}\leq m).$ For $t\in\nn$ let
\beqn
\label{maxw}
W_t=\max_{0\leq i \leq t} Z_i.
\feqn
First, we will establish the following inequality:
\begin{lemma}
\label{lemma1}
For all $\eta>1$ there exists $\theta_\eta>0$ such that for any $\theta\in(0,\theta_\eta]$ and $k,J,m\in\nn$ we have
\beq
P_k(W_m\geq J)\leq e^{\theta\lambda^{-m} (k\eta^m\lambda^m -J)}.
\feq
\end{lemma}
\begin{proof}[Proof of Lemma~\ref{lemma1}]
Let $U_t:=\lambda^{-t}Z_t.$ Then $(U_t)_{t\in\zz_+}$ is a martingale with respect to its natural filtration. For any $\theta>0,$ $f(x)=e^{\theta x}$ is a convex function and hence the sequence $e^{\theta U_t},$ $t\in \zz_+,$ form a sub-martingale. Hence, by Doob's maximal inequality (see, for instance, Theorem~5.4.2 in \cite{durrett}),
\beqn
\nonumber
P_k(W_m\geq J)&=& P_k\bigl(\max_{0\leq k \leq m} e^{\theta U_k}\geq e^{\theta J\lambda^{-m}}\bigr)\leq
e^{-\theta J\lambda^{-m}}E_k\bigl(e^{\theta U_m}\bigr)
\\
\nonumber
&=&
e^{-\theta J\lambda^{-m}}E_k\bigl(E_k\bigl(e^{\theta \lambda^{-m}Z_m}\bigl|Z_{m-1}\bigr)\bigr)
\\
\label{n1}
&=&
e^{-\theta J\lambda^{-m}}E_k\bigl(\exp\bigl(\lambda Z_{m-1}(e^{\theta\lambda^{-m}}-1)\bigr)\bigr).
\feqn
Pick now $\theta >0$ so small that $e^x-1\leq \eta x$ for any positive $x\leq \theta \lambda^{-1}.$ It follows then from \eqref{n1} that for any $J,m\in\nn,$
\beq
P_k(W_m\geq J)\leq e^{-\theta J\lambda^{-m}} E_k\bigl(e^{\theta \eta \lambda^{m-1} Z_{m-1}}\bigr).
\feq
Applying induction, we obtain that
\beq
P_k(W_m\geq J)\leq e^{-\theta J\lambda^{-m}} E_k\bigl(e^{\theta \eta^m Z_0}\bigr)=e^{-\theta J\lambda^{-m}}\cdot e^{\theta \eta^m k},
\feq
as required.
\end{proof}
Recall now the notation $T_0=\inf\{t\in\nn:Z_t=0\}.$ It follows from the results of \cite{agresti74} that for any $k\in\nn,$
\beqn
\label{174}
\Bigl(\frac{qs_1(1-\lambda^mq^m)}{s_1-\lambda^mq^m}\Bigr)^k\leq P(T_0\leq m|Z_0=k)\leq \Bigl(\frac{qs_2(1-\lambda^mq^m)}{s_2-\lambda^mq^m}\Bigr)^k,
\feqn
where $s_1$ and $s_2$ are introduced in \eqref{s74}.
Combining these inequalities with the result of Lemma~\ref{lemma1} for $J=N$ we arrive to the following result:
\begin{lemma}
\label{lemma3}
Let $s_1$ and $s_2$ be defined by \eqref{s74}. Then, for any real $\eta>1$ and integers $k,J,m,N\in\nn,$ the following holds true:
\begin{itemize}
\item [(i)] $P_k(\sigma_{_N}\leq m)\leq \Bigl(\frac{qs_2(1-\lambda^mq^m)}{s_2-\lambda^mq^m}\Bigr)^k+
e^{\theta_\eta \lambda^{-m}(k\eta^m\lambda^m -N)},$ where $\theta_\eta$ is the constant introduced in the statement of Lemma~\ref{lemma1}.
\item[(ii)]  $P_k(\sigma_{_N}\leq m)\geq \Bigl(\frac{qs_1(1-\lambda^mq^m)}{s_1-\lambda^mq^m}\Bigr)^k.$
\end{itemize}
\end{lemma}
Notice that the identity $q=e^{-\lambda(1-q)}$ implies $\lambda q<1$ because $e^{-q^{-1}(1-q)}<q$ for any $q\in (0,1)$ and $e^{-\lambda(1-q)}$
is a decreasing function of $\lambda.$
\par
Recall $\tau_{_N}$ from \eqref{taun}. In view of \eqref{basic} and \eqref{basic1}, in order to complete the proof of Theorem~\ref{main7} it remains to evaluate $P_k(\tau_{_N}\leq m).$
To this end, recall $W_t$ from \eqref{maxw}, fix any $J\in \nn,$ and write using the Markov property of $(\xn_t,Z_t)_{t\in \zz_+}$ and the estimate in \eqref{diff},
\begin{align}
\label{i1}
\nonumber
P_k(\tau_{_N}\leq m)&\leq P_k(\tau_{_N}\leq m~\mbox{\rm and}~W_m< J)+P_k(W_m\geq J)
\\
\nonumber
&\leq
P_k\Bigl(\bigcup_{t=1}^m \bigl\{\xn_{t-1}=Z_{t-1}<J,\,\xn_t\neq Z_t\bigr\}\Bigr)+P_k(W_m\geq J)
\\
\nonumber
&=
\sum_{t=1}^m P_k\bigl(\xn_t\neq Z_t\,\bigr|\,\xn_{t-1}=Z_{t-1}<J\bigr)\cdot P_k\bigl(\xn_{t-1}=Z_{t-1}<J\bigr)
\\
\nonumber
&
\qquad +P_k(W_m\geq J)
\\
\nonumber
&\leq \sum_{t=1}^m P_k\bigl(\xn_t\neq Z_t\,\bigr|\,\xn_{t-1}=Z_{t-1}<J\bigr) +P_k(W_m\geq J)
\\
&\leq m C_0\frac{J^{3/2}}{N} +P_k(W_m\geq J).
\end{align}
Using the result in Lemma~\ref{lemma1} we can deduce from \eqref{i1} the following:
\begin{lemma}
\label{lemma4}
For any real $\eta>1$ and integers $N,m\in \nn,$ $J\in \Omega_{_N}^o,$  we have
\beq
P_k(\tau_{_N}\leq m)\leq m C_0\frac{J^{3/2}}{N}+ e^{\theta(\eta) \lambda^{-m}(k\eta^m\lambda^m -J)},
\feq
where $C_0$ is the constant introduced in \eqref{diff} and $\theta(\eta)$ is the constant $\theta_\eta$ introduced in the statement of Lemma~\ref{lemma1}.
\end{lemma}
The claim of Theorem~\ref{main7} follows now from the bounds in \eqref{basic} and \eqref{basic1} along with the estimates given in Lemma~\ref{lemma3} and Lemma~\ref{lemma4}. \qed
\section{Appendix: Moran process}
\label{amoran}
The goal of this section is to obtain an analogue of Theorem~\ref{main1a} (i.\,e., of the results stated in full detail in Theorem~\ref{main1} and Theorem~\ref{newt}) for the frequency-dependent Moran process introduced in \cite{nature04a,nowakmoran}. The main result of this section is stated in Theorem~\ref{main-moran}.
\par
For a given integer $N\geq 2,$ the Moran process which we denote by $\yn_t,$ $t\in\zz_+,$ is a discrete-time birth and death Markov chain on $\Omega_{_N}$ with transition kernel
$\pn_{i,j}:=$ $P\bigl(\yn_{t+1}=j\bigl|\yn_t=i\bigr)$ defined as follows. The chain has two absorbtion states,
$0$ and $N,$ and for any $i\in\Omega_{_N}^o,$
\beq
\pn_{i,j}=
\left\{
\begin{array}{ll}
\frac{N-i}{N}\xi_{_N}(i)&\mbox{if}~j=i+1\\
\frac{i}{N}\bigl(1-\xi_{_N}(i)\bigr)&\mbox{if}~j=i-1\\
1-\pn_{i,i-1}-\pn_{i,i+1}&\mbox{if}~j=i\\
0&\mbox{otherwise}.
\end{array}
\right.
\feq
The process in this form, with a general selection parameter $w\in (0,1],$  was introduced in \cite{nature04a}. We remark that even though \cite{nowakmoran} formally considered only the basic variant with $w=1,$ their main theorems hold for an arbitrary $w\in (0,1].$
\par
Similarly to \eqref{pni}, we define
\beq
\widehat p_{_N}(i)=P\bigl(\yn_\tau=N|\yn_0=i\bigr),
\feq
where $\tau=\inf\bigl\{t>0:\yn_t=0~\mbox{or}~\yn_t=N\bigr\}.$
Since the Moran model is a birth-death process, the fixation probabilities are known explicitly \cite{nature04a,nowakmoran} (see, for instance,
Example~6.4.4 in \cite{durrett} for a general birth and death chain result):
\beqn
\label{ff}
\widehat p_{_N}(i)=\frac{1+\sum_{j=1}^{i-1} \prod_{k=1}^j \frac{g_{_N}(k)}{f_{_N}(k)}}{1+\sum_{j=1}^{N-1} \prod_{k=1}^j\frac{g_{_N}(k)}{f_{_N}(k)}}.
\feqn
In what follows we however bypass a direct use of this formula.
\par
The result following is an analogue of Lemma~\ref{lem2} for the Moran process.
\begin{lemma}
\label{lem1}
Let Assumption~\ref{assume1} hold. Then for any $N\geq N_0$ and $i\in\Omega_{_N}^o,$
\beqn
\label{e-moran}
E\bigl(\gamma^{\yn_{t+1}}\bigl|\yn_t=i\bigr)\leq \gamma^i \quad \mbox{and} \quad
E\bigl(\alpha^{\yn_{t+1}}\bigr|\yn_t=i\bigr)\geq \alpha^i.
\feqn
\end{lemma}
\begin{proof}[Proof of Lemma \ref{lem1}]
We will only prove the first inequality in \eqref{e-moran}. The proof of the second one can be carried out in a similar manner. We have:
\begin{align*}
&E\bigl(\gamma^{\yn_{t+1}}|\yn_t=i\bigr)=\gamma^{i+1}\frac{N-i}{N}\xi_{_N}(i)+\gamma^{i-1}\frac{i}{N}\bigl(1-\xi_{_N}(i)\bigr)
\\
&
\quad
\quad
+
\gamma^i\Bigl(1-\frac{N-i}{N}\xi_{_N}(i)-\frac{i}{N}\bigl(1-\xi_{_N}(i)\bigr)\Bigr)
\\
&
\quad
=
\gamma^i +
\gamma^{i-1}(1-\gamma)\frac{i}{N}
-
\gamma^{i-1}(1-\gamma)\xi_{_N}(i)\Bigl(\gamma+(1-\gamma)\frac{i}{N}\Bigr).
\end{align*}
Since by virtue of \eqref{assume1} and Assumption~\ref{assume1},
\beq
\frac{i}{N}-\xi_{_N}(i)\Bigl(\gamma+(1-\gamma)\frac{i}{N}\Bigr)\leq 0,
\feq
we conclude that $E\bigl(\gamma^{\yn_{t+1}}\bigl|\yn_t=i\bigr)\leq\gamma^i.$
\end{proof}
In the same way as Lemma~\ref{lem2} implies Theorem~\ref{main1}, the above result yields the following
bounds for the fixation probabilities in the Moran process:
\beqn
\label{mmmbounds}
\frac{1-\gamma^i}{1-\gamma^N}\leq \widehat p_{_N}(i) \leq \frac{1-\alpha^i}{1-\alpha^N}.
\feqn
More precisely, we have:
\begin{theorem}
\label{main-moran}
Let Assumption~\ref{assume1} hold. Then:
\begin{itemize}
\item [(i)] For any $N\geq N_0$ and $i\in\Omega_{_N}^o,$ the inequalities in \eqref{mmmbounds} hold true.
\item [(ii)] Furthermore, for any $i\in\nn,$
\beqn
\label{result3}
\lim_{N \to \infty} \widehat p_{_N}(i)=1-\lambda^{-i},
\feqn
where $\lambda$ is defined in \eqref{lambda}.
\end{itemize}
\end{theorem}
A proof of the limit in \eqref{result3} is outlined at the end of the appendix, after the following Remark.
\begin{remark}
\label{re}
The identities for the optimal values of $\alpha$ and $\gamma$ given in \eqref{alphagamma} suggest that in some cases one of the bounds
in \eqref{mmmbounds} might be asymptotically tight for large populations.
The purpose of this remark is to explore conditions for the equalities $\lambda^{-1}=\alpha$ or $\lambda^{-1}=\gamma$ to hold true.
By the definition given in \eqref{lambda}, $\lambda=\lim_{N\to\infty} \frac{f_{_N}(i)}{g_{_N}(i)}$ for any $i\in\nn.$ Moreover, for any $N\geq N_0,$
\beqn
\label{e}
\frac{f_{_N}(1)}{g_{_N}(1)}=\frac{1-w+wb}{1-w+wd+(c-d)(N-1)^{-1}}
\feqn
and
\beqn
\label{e1}
\frac{f_{_N}(N-1)}{g_{_N}(N-1)}=\frac{1-w+wa+(b-a)(N-1)^{-1}}{1-w+wc}.
\feqn
It follows from \eqref{e} that
\beq
\lambda=\frac{1-w+wb}{1-w+wd}=
\left\{
\begin{array}{cl}
\sup\limits_{N\geq N_0} \frac{f_{_N}(1)}{g_{_N}(1)}&\mbox{\rm if}~c\geq d\\
\inf\limits_{N\geq N_0} \frac{f_{_N}(1)}{g_{_N}(1)}&\mbox{\rm if}~c\leq d.
\end{array}
\right.
\feq
Furthermore, \eqref{e1} implies that
\beq
\frac{1-w+wa}{1-w+wc}=
\left\{
\begin{array}{cl}
\inf\limits_{N\geq N_0} \frac{f_{_N}(N-1)}{g_{_N}(N-1)}&\mbox{\rm if}~b\geq a\\
\sup\limits_{N\geq N_0} \frac{f_{_N}(N-1)}{g_{_N}(N-1)}&\mbox{\rm if}~b\leq a
\end{array}
\right.
\feq
and
\beq
\frac{1-w+wa+(b-a)(N_0-1)^{-1}}{1-w+wc}=
\left\{
\begin{array}{cl}
\max\limits_{N\geq N_0} \frac{f_{_N}(N-1)}{g_{_N}(N-1)}&\mbox{\rm if}~b\geq a\\
\min\limits_{N\geq N_0} \frac{f_{_N}(N-1)}{g_{_N}(N-1)}&\mbox{\rm if}~b\leq a.
\end{array}
\right.
\feq
In principle, the last three identities contain all the information which is needed to identify necessary and sufficient conditions for the occurrence
of either $\alpha=\lambda^{-1}$ or $\gamma=\lambda^{-1},$ where it is assumed, as in \eqref{alphagamma}, that the optimal bounds
$\alpha=\inf_{N\geq N_0}\min_{i\in \Omega^o_{_N}} \frac{g_{_N}(i)}{f_{_N}(i)}$ and $\gamma=\sup_{N\geq N_0}\max_{i\in \Omega^o_{_N}} \frac{g_{_N}(i)}{f_{_N}(i)}$ are employed. For instance, in the generic prisoner's dilemma case $b>d>a>c$ one can set
\beqn
\label{lg}
\gamma^{-1}=\lambda=\inf_{N\geq N_0}\min_{i\in \Omega^o_{_N}} \frac{f_{_N}(i)}{g_{_N}(i)}
\feqn
provided that $\fracd{1-w+wb}{1-w+wd} \leq \fracd{1-w+wa}{1-w+wc},$ which is equivalent to
\beq
(1-w)(c+b-a-d)\leq w(ad-bc).
\feq
This leads us to consider the following possible scenarios for a prisoner's-dilemma-type underlying game, that is
assuming that $b>d>a>c:$
\begin{enumerate}
\item $ad\geq bc$ and $c+b-a-d\leq 0$ (for instance, $b=4,d=3,a=2,c=1$). In this case \eqref{lg} holds for any $w\in (0,1].$
\item $ad>bc$ and $c+b-a-d> 0$ (for instance, $b=5,d=3,a=2,c=1$). In this case \eqref{lg} holds if and only if
\beq
\frac{1-w}{w}\leq \frac{ad-bc}{c+b-a-d}
\quad
\Leftrightarrow
\quad
w\geq \Bigl(1+\frac{ad-bc}{c+b-a-d}\Bigr)^{-1}.
\feq
\item $ad=bc$ and $c+b-a-d> 0$ (for instance, $b=6,d=3,a=2,c=1$). In this case \eqref{lg} holds only if $w=1.$
\item $ad<bc$ and $c+b-a-d\geq 0$ (for instance, $b=7,d=3,a=2,c=1$). In this case \eqref{lg} holds for no $w\in (0,1].$
\item It remains to consider the case when $ad<bc$ and $c+b-a-d< 0.$ We will now verify that this actually cannot happen.
To get a contradiction, assume that this scenario
is feasible and let $\veps=\min\{b-d,a-c\},$ $d'=d+\veps,$ $c'=c+\veps.$   Then
\beqn
\label{r1}
c'+b-a-d'=c+b-a-d< 0
\feqn
and, since $ad<bc$ and $a<b,$
\beqn
\label{r3}
ad'=ad+a\veps<bc'=bc+b\veps.
\feqn
But, due to the choice of $\veps$ we made, we should have either $b=d'$ or $a=c'.$ In the former case
\eqref{r1} and \eqref{r3} imply the combination of inequalities $c'-a< 0$ and $a<c',$ while in the latter they yield
$b-d'< 0$ and $d'<b,$ neither of which is possible.
\end{enumerate}
\end{remark}
The limit in \eqref{result3} has been computed in \cite{antalsm} (technically, in the specific case $w=1$), see in particular formula (39) there.
The proof in \cite{antalsm} relies on \eqref{ff} and involves some semi-formal approximation arguments.
We conclude this appendix with the outline of a formal proof of this result which is based on a different approach, similar to the one employed
in the proof of Theorem~\ref{newt}.
\par
Toward this end observe first that, provided that both processes have the same initial state, the fixation probabilities
of the Markov chain $\yn$ coincide with those of a Markov chain $\tyn=\bigl(\tyn_t\bigr)_{t\in\zz_+}$ on the state space $\Omega_{_N}$ with transition kernel $\tpn_{i,j}:=$ $P\bigl(\tyn_{t+1}=j\bigl|\tyn_t=i\bigr)$ which is defined as follows. Similarly to $\yn,$ the chain $\tyn$ has two absorbtion states,
$0$ and $N.$ Furthermore, for any $i\in\Omega_{_N}^o,$
\beq
\tpn_{i,j}=
\left\{
\begin{array}{cl}
\frac{\pn_{i,i+1}}{\pn_{i,i-1}+\pn_{i,i+1}}&\mbox{if}~j=i+1\\
\frac{\pn_{i,i-1}}{\pn_{i,i-1}+\pn_{i,i+1}}&\mbox{if}~j=i-1\\
0&\mbox{otherwise}.
\end{array}
\right.
\feq
The advantage of using the chain $\tyn$ over $\yn$ rests on the fact that while both $\pn_{i,i-1}$ and $\pn_{i,i+1}$ converge to zero as $N\to\infty,$
\beq
\lim_{N\to\infty} \tpn_{i,i+1}&=&\lim_{N\to\infty} \frac{\frac{N-i}{N}\xi_{_N}(i)}{\frac{N-i}{N}\xi_{_N}(i)+
\frac{i}{N}\bigl(1-\xi_{_N}(i)\bigr)}=\lim_{N\to\infty} \frac{\xi_{_N}(i)}{\xi_{_N}(i)+
\frac{i}{N}}
\\
&=&
\frac{\lambda}{1+\lambda}
\feq
and, consequently, $\lim_{N\to\infty} \tpn_{i,i-1}=\frac{1}{1+\lambda}.$ In other words, as $N\to\infty,$ the sequence of Markov chains $\tyn$ converges weakly
to the nearest-neighbor random walk (birth and death chain) $\witi Z=\bigl(\witi Z_t\bigr)_{t\in\zz_+}$ on $\zz_+$ with absorbtion state at zero and transition kernel defined at $i\in\nn$ as follows:
\beq
P\bigl(\witi Z_{t+1}=j\bigl|\witi Z_t=i)=
\left\{
\begin{array}{cl}
\frac{\lambda}{1+\lambda}&\mbox{if}~j=i+1\\
\frac{1}{1+\lambda}&\mbox{if}~j=i-1\\
0&\mbox{otherwise}.
\end{array}
\right.
\feq
Since $\lambda>1,$ then, similarly to \eqref{ztl}, we have $P\bigl(\lim_{t\to\infty} \witi Z_t=+\infty\bigr)>0.$
The rest of the proof is similar to the argument following \eqref{ztl} in the proof of Theorem~\ref{newt}, with
the processes $\tyn$ and $\witi Z$ considered instead of, respectively, $\xn$ and $Z.$ The only two exceptions are:
\\
1. $P\bigl(\lim_{t\to\infty} \witi Z_t=0\bigl|\witi Z_t=i\bigr)=\lambda^i$ for any $i\in\nn.$ This follows from the solution
to the ``infinite-horizon" variation of the standard gambler's ruin problem \cite{durrett} (take the limit as $N\to\infty$ in \eqref{ff}
assuming that $\frac{g_{_N}(k)}{f_{_N}(k)}=\lambda^{-1}$ for all $N,k\in\nn$).
\\
2.
The last four lines in \eqref{compom} should be suitably replaced. For instance, one can use the following bound:
\begin{align*}
&
P\Bigl(\tyn_1\neq 0,\ldots, \tyn_{K-1}\neq 0,\,\max_{0\leq t \leq K-1}\tyn_t<m\Bigr)
\\
&
\qquad\qquad
+\sum_{j=m}^{N-1} P\Bigl(\max_{0\leq t \leq K-1}\tyn_t=j\Bigr)\cdot \bigl(1-\widehat p_{_N}(j)\bigr)
\\
&
\quad
\leq
\Bigl[1-P\bigl(\tyn_1=m-2,\tyn_2=m-3,\ldots, \tyn_{m-1}=0\bigl|\tyn_0=m-1\bigr)\Bigr]^{\frac{K}{m-1}}
\\
&
\qquad\qquad
+
\bigl(1-\widehat p_{_N}(m)\bigr)
\\
&
\quad
\leq
\Bigl[1-\prod_{i=1}^{m-1}\frac{\pn_{i,i-1}}{\pn_{i,i-1}+\pn_{i,i+1}}\Bigr]^{\frac{K}{m-1}}+\frac{\gamma^m}{1-\gamma^N}
\\
&
\quad
=\Bigl[1-\prod_{i=1}^{m-1}\frac{g_{_N}(i)}{g_{_N}(i)+f_{_N}(i)}\Bigr]^{\frac{K}{m-1}}+\frac{\gamma^m}{1-\gamma^N}
\\
&
\quad
\leq
\Bigl[1-\Bigl(\frac{\alpha}{1+\alpha}\Bigr)^{m-1}\Bigr]^{\frac{K}{m-1}}+\frac{\gamma^m}{1-\gamma^N}.
\end{align*}
We leave the details to the reader.

\section*{Acknowledgements}

\noindent The work of T. C. was partially supported by the
Alliance for Diversity in Mathematical Sciences Postdoctoral Fellowship. O. A. thanks the Department of Mathematics at Iowa State University for its hospitality during a visit in which part of this work was carried out. A.M. would like to thank the Computational Science and Engineering Laboratory at ETH Z\"urich for the warm hospitality during a sabbatical semester.
The research of A.M. is supported in part by the National Science Foundation under Grants NSF CDS\&E-MSS 1521266 and NSF CAREER 1552903.






\end{document}